\documentclass[12pt]{amsart}
\usepackage{amsfonts,amssymb,amsmath,amscd,amsthm, amstext,bm,color}
\usepackage[colorlinks, citecolor=blue,pagebackref,hypertexnames=false]{hyperref}
\usepackage{enumerate}
\usepackage{epstopdf}
\usepackage{fancyhdr}
\usepackage{geometry}
\usepackage{graphicx}
\usepackage{mathrsfs}
\usepackage{txfonts}
\usepackage{tikz}
\usepackage{url}

\usepackage{comment}
\allowdisplaybreaks

\newtheorem{theorem}{Theorem}[section]

\newtheorem{lemma}[theorem]{Lemma}

\newtheorem{remark}[theorem]{Remark}

\newtheorem{definition}[theorem]{Definition}

\numberwithin{equation}{section}

\begin{document}


\title{The fractional Riesz transform and their commutator in Dunkl setting}

\author[Y.P. Chen, X.T. Han and L.C. Wu]{Yanping Chen,
 Xueting Han$^*$,
 Liangchuan Wu
}

\thanks{$^*$ Corresponding author}

\address{Yanping Chen, Department of Mathematics, Northeastern University, Shenyang, 110004, Liaoning, China}
\email{yanpingch@126.com}

\address{Xueting Han, School of Mathematics and Physics, University of Science and Technology Beijing, Beijing, 100083, China}
\email{hanxueting12@163.com}

\address{Liangchuan Wu, School of Mathematical Science, Anhui University, Hefei, 230601, Anhui, China}
\email{wuliangchuan@ahu.edu.cn}



\thanks{
}


\begin{abstract}
In this paper, we study the boundedness of the fractional Riesz transforms
in the Dunkl setting. Moreover, we establish the necessary and sufficient
conditions for the boundedness of their commutator with respect to the central BMO
space associated with Euclidean metric and the BMO space associated with Dunkl metric, respectively.
Based on this, we further characterize the compactness of the commutator in terms of  the corresponding types of
VMO spaces.
\end{abstract}

\maketitle

\section{Introduction}

Fourier transform in $\mathbb R^N$
$$
    \hat{f}(\xi)=\int_{\mathbb{R}^N}f(x)e^{\langle x,-i\xi\rangle} dx
$$
 plays a crucial role in classic analysis, especially providing a powerful tool in the study of the Riesz transforms. The classic Riesz transforms $\mathfrak{R}_j $, $j=1,\ldots, N$,
can be expressed in the frequency domain as Fourier multipliers:
$$
    \widehat{\mathfrak{R}_j f}(\xi) = i \frac{\xi_j}{\|\xi\|} \widehat{f}(\xi),
$$
where $\|\cdot\|$ denotes the Euclidean norm in $\mathbb{R}^N$.
This representation reduces the analysis of their \( L^p \)-boundedness to verify the boundedness of the symbol \( i \xi_j / \|\xi\| \) due to the Plancherel theorem.
Furthermore, the Fourier transform reveals deep connections between Riesz transforms and derivatives, highlighting their vital role in the study of Sobolev spaces, Hardy spaces, and elliptic partial differential equations (see \cite{St2}).

The classical fractional Riesz transform, defined as
$$
    \mathfrak{R}_j^{\alpha} f(x) = c_{N,\alpha} \,  \int_{\mathbb{R}^N} \frac{y_j}{\|y\|^{N-\alpha+1}} f(x-y) \,dy
$$
for $0<\alpha<N$, has been extensively studied.
Especially,
 the vector of the fractional Riesz transform
$$
    \mathfrak{R}^{\alpha} =\big\{\mathfrak{R}_1^{\alpha},
    \mathfrak{R}_2^{\alpha},\ldots,\mathfrak{R}_N^{\alpha}\big\}
$$
appears in the generalized surface quasi-geostrophic (SQG) equation:
$$
\begin{cases}
w_{t}+u \cdot \nabla w=0, \quad(x, t) \in \mathbb{R}^{2} \times  \mathbb{R}^{+},\\
u=\nabla^{\perp}(-\Delta)^{-1+\beta} w, \\
w(x, 0)=w_{0},
\end{cases}
$$
where $0 \leq \beta\leq \frac{1}{2}$. The unknown functions $w=w(x, t)$ and $u=u(x, t)=\left(u_{1}(x, t), u_{2}(x, t)\right)$ related by the second equation in the above equations can be expressed as
\begin{equation}\label{eq13}
    u(x)=\bigg(-\int_{\mathbb{R}^{2}} \frac{y_2}{\|y\|^{2+2 \beta}} w(x-y) \,d y,\   \
    \int_{\mathbb{R}^{2}} \frac{y_1}{\|y\|^{2+2 \beta}} w(x-y) \,d y\bigg).
\end{equation}
For  $0<\beta<1 / 2$, \eqref{eq13} is completely similar to
$\mathfrak{R}^{1-2\beta} =\big \{\mathfrak{R}_1^{1-2\beta},\mathfrak{R}_2^{1-2\beta}\big \}$ with $N=2$.
The two-weight inequalities of $\mathfrak{R}^{\alpha}$ were characterized by Lacey, Sawyer, Wick et al. in \cite{LaWi13, SaShUr2016, SaShUr2020, SaWi2025}.

Along with the introduction of a parallel theory to the Fourier transform, the Dunkl transform, another fundamental tool has been developed on Euclidean spaces over the past several decades
 (see for example \cite{ABFR, AH, ADH, AT, BCG, BCV, BR, Dzi2016, DA2023, DH2022}).
The Dunkl transform was introduced by Dunkl \cite{Dunkl} under the action of a reflection group. Specifically, the reflection $\sigma_\alpha$ with respect to the hyperplane
orthogonal to a nonzero vector $\alpha$ is given by
$$
    \sigma_\alpha (x)=x-2\frac{\langle x,\alpha\rangle}{\|\alpha\|^2}\alpha.
$$
A finite set $R\subseteq\mathbb{R}^N\setminus\{0\}$ is called a root system if $\sigma_\alpha (R)=R$ for every $\alpha\in R$. We shall consider normalized root systems in this work, meaning that $\langle\alpha,\alpha\rangle=2$ for all $\alpha\in R$.
The finite group $G$ generated by the set of reflections $\{\sigma_\alpha: \alpha \in R\}$, where
 $\sigma_\alpha(x)=x-\langle\alpha, x \rangle\alpha$ for any $ x\in \mathbb{R}^N$,
 is called the reflection group of the root system. For any $x\in\mathbb{R}^N$, we denote by
$$
    \mathcal{O}(x)=\big\{\sigma (x): \sigma\in G\big\}
$$
 the $G$-orbit of the point $x$.
Then the Dunkl metric $d$, which denotes the distance between two $G$-orbits $\mathcal O(x)$ and $\mathcal O(y)$, is defined by
$$
    d(x,y):=\min_{\sigma \in G} \|x-\sigma(y)\|.
$$
 It is straightforward to see $d(x,y)\leq\|x-y\|$.

 Given a root system $R$ and a fixed multiplicity function  $\kappa$ (defined on $R$) which is a nonnegative $G$-invariant function, the $G$-invariant homogeneous weight function $h_{\kappa}$ is defined as
$$
    h_{\kappa}(x)=\prod_{\alpha\in R}|\langle \alpha, x\rangle|^{\kappa(\alpha)}.
$$
The associated Dunkl measure is then given by
$$
    d\omega(x):=h_{\kappa}(x)dx=\prod_{\alpha\in R}|\langle \alpha, x\rangle|^{\kappa(\alpha)}
dx.
$$
Let $\gamma_{\kappa}=\sum_{\alpha\in R}\kappa(\alpha)$, $\mathbf{N}=N+\gamma_{\kappa}$ is the homogeneous dimension associated with the Dunkl setting.

The Dunkl transform is defined by
$$
    \mathcal{F}_{\kappa} f(\xi)=c_{\kappa}^{-1}\int_{\mathbb{R}^N}f(x)E(x,-i\xi) d\omega(x).
$$
Here, $\displaystyle  c_{\kappa}=\int_{\mathbb{R}^N}e^{-\|x\|^2/2}d\omega(x)$ and the function $E(x,y)$ on
$\mathbb{C}^N \times \mathbb{C}^N$
is the Dunkl kernel which generalizes the exponential function $e^{\langle x,y\rangle}$ in the Fourier transform.
 There also exists a Dunkl translation $\tau$
 which serves as an analogue to the ordinary translation $\tau_x f(\cdot)=f(\cdot-x)$.

Thangavelu and Xu \cite{TX2007} introduced the Riesz transforms $R_j, \, j=1,2, \ldots, N$, in the Dunkl setting.
It was shown to be a multiplier operator via the Dunkl transform
and  (for $N=1$) is bounded on $L^p(\mathbb{R}, d\omega)$ for $1<p<\infty$  (see \cite[Theorems 5.3 and  5.5]{TX2007}).  This boundedness was later
extended to $L^p(\mathbb{R}^N, d\omega)$ in \cite{AS}.

In this article, we focus on the fractional Riesz transform in the Dunkl setting,  defined via the Dunkl translation:
\begin{equation}\label{Riesztrans}
    R_j^{\alpha}(f)(x)=d_{\kappa,\alpha}   \int_{ \mathbb{R}^N}  \tau_x(f)(-y) \frac{y_j}{\|y\|^{\mathbf{N}+1-\alpha}} d \omega(y),
    \quad x \in \mathbb{R}^N,
\end{equation}
where $\displaystyle  d_{\kappa,\alpha}=2^{\frac{\mathbf{N}-\alpha}{2}} \Gamma\Big(\frac{\mathbf{N}+1-\alpha}{2}\Big )/\Gamma\Big (\frac{1+\alpha}{2}\Big )$ and $0\leq \alpha<\mathbf{N}$.
In fact, $R_j^{\alpha}$ is a convolution operator and can also be defined as
$$
    \mathcal{F}_{\kappa}(R_j^{\alpha} f)(\xi)=-i\frac{\xi_j}{\|\xi\|^{1+\alpha}} \mathcal{F}_{\kappa}(f)(\xi).
$$
Note that $R_j^\alpha$ reduces to the Dunkl Riesz transform $R_j$ when $\alpha =0$.

Our first goal is to give the $(L^p, L^q)$-boundedness for $R_j^{\alpha}$ for functions $f \in L^p(\mathbb{R}^N,d\omega)$ with $\displaystyle  \frac{1}{q}=\frac{1}{p}-\frac{\alpha}{\mathbf{N}}$ when $0<\alpha<N$. The result is stated as follows.

\begin{theorem}\label{operatorbound}
 Given $0<\alpha<N$ and $\displaystyle  1<p<\frac{\mathbf{N}}{\alpha}$.
 Let $\displaystyle  \frac{1}{q}=\frac{1}{p}-\frac{\alpha}{\mathbf{N}}$.
 Then $ R_j^{\alpha}$ is bounded from $L^p(\mathbb{R}^N, d \omega)$ to $L^q(\mathbb{R}^N, d \omega)$ with
$$
    \|R_j^{\alpha}f\|_{L^q(\mathbb{R}^N, d \omega) } \lesssim  \|f\|_{L^p(\mathbb{R}^N, d \omega) }.
$$
\end{theorem}

We continue to consider the commutator of the fractional Dunkl Riesz transform, which is defined by
$$
    [b, R_j^{\alpha}](f)(x)=b(x) R_j^{\alpha}(f)(x)-R_j^{\alpha}(b f)(x)
$$
for functions $b\in L_{l o c}^1(\mathbb{R}^N, d \omega)$.

To investigate the properties of these commutators, we introduce certain types of spaces of bounded mean oscillation in the Dunkl setting. A function $b\in L_{l o c}^1(\mathbb{R}^N, d \omega)$ is said to belong to the
$\mathrm{BMO}_{\text{Dunkl}}$ space if its norm satisfies
\begin{align}\label{eqBMODef}
    \|b\|_{\mathrm{BMO}_{\text {Dunkl}}}=\sup _{B\subseteq \mathbb{R}^N} \frac{1}{\omega(B)} \int_{B}|b(x)-b_{B}| d \omega(x)<\infty
\end{align}
with the supremum taking over all the Euclidean balls, and
$$
    b_{B}=\frac{1}{\omega(B)} \int_{B} b(x) d \omega(x).
$$
We also define the $\mathrm{BMO}_d$ space associated with the Dunkl metric $d$ by the set of the functions
$b\in L_{loc}^1(\mathbb{R}^N, d \omega)$ satisfying
$$
    \|b\|_{\mathrm{BMO}_d}=\sup _{B \in \mathbb{R}^N}
    \frac{1}{\omega(\mathcal{O}(B))} \int_{\mathcal{O}(B)}|b(x)-b_{\mathcal{O}(B)}| d \omega(x)<\infty,
$$
where $\mathcal{O}(B)$ denotes the Dunkl ball:
$$
    \mathcal{O}(B(x,r)):=\big\{y\in \mathbb{R}^N:d(y,x)<r\big\}.
$$
Note that $\mathrm{BMO}_d\subsetneqq \mathrm{BMO}_{\text{Dunkl}}$ (see \cite{JL2023}).

Han et al. \cite{HLLW2023} established  the $(L^p, L^q)$-boundedness for the commutator of the fractional operator when $b\in\mathrm{BMO}_{\text{Dunkl}}$. In \cite{HLLW}, the same authors studied the lower and the upper bounds of the commutator of the Dunkl Riesz transform with respect to the $\mathrm{BMO}_{\text{Dunkl}}$ and $\mathrm{BMO}_d$ spaces, respectively. Moreover, they characterized the compactness of these commutators in terms of
 two types of vanishing mean oscillation spaces, specifically the subspaces of $\mathrm{BMO}_{\text{Dunkl}}$ and $\mathrm{BMO}_d$ spaces. It should be addressed that the authors of \cite{HLLW} achieved these results by establishing the pointwise smoothness estimates for the kernel of the Dunkl Riesz transform.

Motivated by this, we will provide the upper bound for the commutator $[b, R_j^{\alpha}]$
via $\mathrm{BMO}_d$ space. To describe its lower bound, we introduce a subspace of $\mathrm{BMO}_{\text{Dunkl}}$, called the central BMO space and denoted by $\mathrm{CBMO}_{\text{Dunkl}}$. Concretely,
\begin{align*}
  \mathrm{CBMO}_{\text{Dunkl}}&=\big\{b\in \mathrm{BMO}_{\text{Dunkl}}:  \text{the supremum in \eqref{eqBMODef} is taken over  }   \big.\\
 &\hspace{3.5cm}\big. \text{all Euclidean balls } B\subseteq \mathbb{R}^N  \text{ that contain the origin $0$}\big\}.
\end{align*}
We now present our results as follows.

\begin{theorem}\label{commutatorbound}
Given $0<\alpha<N$ and $\displaystyle  1<p<\frac{\mathbf{N}}{\alpha}$.
 Let $\displaystyle  \frac{1}{q}=\frac{1}{p}-\frac{\alpha}{\mathbf{N}}$.
 Suppose $b \in L_{l o c}^1(\mathbb{R}^N, d \omega)$. Consider the commutator $[b, R_j^{\alpha}]$. Suppose $b \in \mathrm{BMO}_d$. Then $[b, R_j^{\alpha}]$ is bounded from  $L^p(\mathbb{R}^N, d \omega)$ to to $L^q(\mathbb{R}^N, d \omega)$, with
$$
\big \|[b, R_j^{\alpha}]\big\|_{L^p(\mathbb{R}^N, d \omega) \rightarrow L^q(\mathbb{R}^N, d \omega)} \lesssim\|b\|_{\mathrm{BMO}_d}.
$$

Conversely, if $[b, R_j^{\alpha}]$ is bounded from $L^p(\mathbb{R}^N, d \omega)$ to $L^q(\mathbb{R}^N, d \omega)$, then $b \in \mathrm{CBMO}_{\text {Dunkl}}$ with
$$
\|b\|_{\mathrm{CBMO}_{\text {Dunkl}}} \lesssim\big\|[b, R_j^{\alpha}]\big \|_{L^p(\mathbb{R}^N, d \omega) \rightarrow L^q (\mathbb{R}^N, d \omega)}
$$
\end{theorem}

 With the characterization of the boundedness of the commutators in hand, we tend to explore their additional properties. In particular, we study sufficient and necessary conditions for the compactness of the commutators using the vanishing mean oscillation spaces $\mathrm{VMO}_d$ and $\operatorname{CVMO}_{\text{Dunkl}}$ which are the subspaces of $\mathrm{BMO}_{d}$ with the Dunkl metric and $\mathrm{CBMO}_{\text{Dunkl}}$ space with Euclidean metric, respectively.

\begin{theorem}\label{commutatorcompact}
Given $0<\alpha<N$ and $\displaystyle 1<p<\frac{\mathbf{N}}{\alpha}$.
 Let $\displaystyle \frac{1}{q}=\frac{1}{p}-\frac{\alpha}{\mathbf{N}}$. If $b \in \mathrm{VMO}_d$. Then $[b, R_j^{\alpha}]$ is compact from $L^p(\mathbb{R}^N, d \omega)$ to $L^q(\mathbb{R}^N, d \omega)$. If $[b, R_j^{\alpha}]$ is compact from $L^p(\mathbb{R}^N, d \omega)$ to $L^q(\mathbb{R}^N, d \omega)$. Then $b \in \operatorname{CVMO}_{\text {Dunkl }}$.
\end{theorem}


This paper thoroughly explores the convolution kernels of the fractional Riesz transforms and offers their pointwise lower and upper size estimates, as well as smoothness conditions, employing the method in \cite{HLLW}. These estimates not only derive the boundedness of the fractional Riesz transforms and their commutators, but also provide the tool for further investigations into compactness properties.
Moreover, considering that the Dunkl measure of an Euclidean ball cannot always be characterized by its homogeneous dimension, we choose the central Dunkl BMO space related to the Euclidean norm to investigate the lower bound of the commutator. It remains challenging to extend this subspace to the entire Dunkl BMO space associated with both the Euclidean metric and the Dunkl metric.

 This paper is organized as follows.
 In the next section, we will recall
  some basic definitions and present some useful lemmas.
 In Section \ref{proofoperator}, we give the proof of
  Theorem \ref{operatorbound}.
 The upper and lower bounds for the commutator $[b, R_j^{\alpha}]$ will be discussed
 in Section \ref{proofcomubound}.
  In the last section, we provide the proof of Theorem \ref{commutatorcompact}.

To simplify the notations throughout this paper,
we write $X \lesssim Y$ to indicate the existence of a constant
 $C$ such that $X \leq CY$.
 Positive constants may vary across different occurrences. If we write $X \approx Y$,
 then both $X \lesssim Y$ and $Y \lesssim X$ hold. From the next section, we will use the notation $\|f\|_{p}=\|f\|_{L^p(\mathbb{R}^N, d \omega)}$ for any $1\leq p \leq \infty$.

\medskip

\section{Preliminaries}\label{preliminaries}

In this section, we first introduce some basic definitions and results in the Dunkl settings. For details we refer the reader to \cite{Dunkl, Roesler3, Roesler-Voit}.

In the Dunkl setting, for any Euclidean ball $B(x,r)=\{y\in \mathbb{R}^N: \|x-y\|<r\}$ centred at $x$ with radius $r$,
the scaling property
$$
 \omega(B(tx, tr))=t^{\mathbf{N}}\omega(B(x,r))
$$
holds for all $x \in \mathbb{R}^N$, $t,r>0$ and the number $\mathbf{N}$ is called the homogeneous dimension.

Observe that
\begin{equation}\label{ballmeasuresim}
\omega(B(x,r))\approx r^{N}\prod_{\alpha\in R}(\,|\langle\alpha,x\rangle|+r\,)^{\kappa(\alpha)}.
\end{equation}
Thus the measure $\omega$ satisfies the doubling condition, that is, there is a constant $C>0$ such that
\begin{equation*}\label{eqdoubling}
\omega(B(x,2r))\leq C\,\omega(B(x,r)).
\end{equation*}
It implies from the doubling condition that
\begin{equation}\label{equival vollum}
\omega(B(x,r))\approx \omega(B(y,r)), \text{ if } \|x-y\|\approx r.
\end{equation}
Moreover, $\omega$ is also a reverse doubling measure. There exists a constant $C\ge1$ such that,
for every $x\in\mathbb{R}^N$ and for every $r_1\geq r_2>0$,
\begin{equation}\label{growth}
C^{-1}\Big(\frac{r_1}{r_2}\Big)^{N}\leq\frac{{\omega}(B(x,r_1))}{{\omega}(B(x,r_2))}\leq C \Big(\frac{r_1}{r_2}\Big)^{\mathbf N}.
\end{equation}

 The ball defined via the Dunkl metric $d$ is
 $$\mathcal{O}(B(x,r)):=\left\{y\in \mathbb{R}^N:d(y,x)<r\right\}=\bigcup_{\sigma \in G}B(\sigma(x),r).
 $$
Since $G$ is a finite group, we have
$$
\omega(B(x,r))\leq \omega(\mathcal{O}(B(x,r)))\leq |G|\,{\omega}(B(x,r)).
$$
Combining with \eqref{equival vollum}, we have
\begin{equation}\label{equival O vollum}
\omega(\mathcal{O}(B(x,r)))\approx \omega(\mathcal{O}(B(y,r))), \text{ if } d(x,y)\approx r.
\end{equation}
Set
$$
V(x,y,t)=\max\{\omega(B(x,t)), \omega(B(y, t))\}.
$$

\smallskip

{\bf Dunkl operator.}
Given the reflection group $G$ of a root system $R$ and a fixed nonnegative multiplicity function $\kappa$.
$R^+$ is a positive subsystem of $R$ where the elements span a cone in the space of roots.
The Dunkl operators $T_\xi$ introduced in \cite{Dunkl} are defined by the following difference operators:
\begin{align*}
T_\xi f(x)
&{=\partial_\xi f(x)+\sum_{\alpha\in R}\frac{\kappa(\alpha)}2\langle\alpha,\xi\rangle
\frac{f(x)-f(\sigma_\alpha(x))}{\langle\alpha,x\rangle}}\\
&=\partial_\xi f(x)+\sum_{\alpha\in R^+}\kappa(\alpha)\langle\alpha,\xi\rangle\frac{f(x)-f(\sigma_\alpha(x) )}{\langle\alpha,x\rangle},
\end{align*}
which are the deformations of the directional derivatives $\partial_\xi$.

\smallskip

{\bf Dunkl kernel.} For fixed ${x}\in\mathbb{R}^N$, consider the simultaneous eigenfunction problem
$$
T_\xi f=\langle x,\xi\rangle\,f,\quad\forall\;\xi\in\mathbb{R}^N.
$$
 Then, its unique solution $f(y)=E(x,y)$ with $f(0)=1$ is the Dunkl kernel ${y\longmapsto}{E}(x,y)$.
The following integral formula was {obtained} by R\"osler \cite{Roesle99}\,{:}
\begin{equation}\label{Rintegral}
{E}(x,y)=\int_{\mathbb{R}^N} e^{\langle\eta, y\rangle} d\mu_{x}(\eta),
\end{equation}
where ${\mu_{x}}$ is a probability measure supported {in the convex hull} $\operatorname{conv}\mathcal{O}(x)$ of the $G$-orbit of $x$. The function ${E}(x,y)$ extends holomorphically to $\mathbb{C}^N\times\mathbb{C}^N$. Please refer to \cite{ADH} for more properties for the Dukl kernel.

\smallskip

{\bf Dunkl transform and Dunkl translation.}
The Dunkl transform is a topological automorphism of the Schwartz space $\mathcal{S}(\mathbb{R}^N)$. For every $f\in\mathcal{S}(\mathbb{R}^N)$ and actually for every $f\in L^1({\mathbb{R}^N, d\omega})$ such that $\mathcal{F}_{\kappa}f\in L^1({\mathbb{R}^N, d\omega})$, we have
$$
f(x)=\big(\mathcal{F}_{\kappa}\big)^2f(-x),
\quad{\forall\;x\in\mathbb{R}^N}.
$$
Moreover, the Dunkl transform extends to an isometric automorphism of $L^2({\mathbb{R}^N, d\omega})$ (see \cite{dJ}, \cite{Roesler-Voit}).

The Dunkl translation $\tau_{x}f$ of a function $f\in\mathcal{S}(\mathbb{R}^N)$ by $x\in\mathbb{R}^N$ is defined by
\begin{equation*}
\tau_{x} f(y)=c_{\kappa}^{-1} \int_{\mathbb{R}^N}{E}(i\xi,x)E(i\xi,y)\mathcal{F}_{\kappa}f(\xi)  \,d\omega(\xi).
\end{equation*}
Notice that each translation $\tau_{x}$ is a continuous linear map of $\mathcal{S}(\mathbb{R}^N)$ into itself, which extends to a contraction on $L^2({\mathbb{R}^N, d\omega})$. The Dunkl translations $\tau_{x}$ and the Dunkl operators $T_\xi$ all commute. For all $x,y\in\mathbb{R}^N$, and
$f,g\in\mathcal{S}(\mathbb{R}^N)$, $\tau_{x}$ also satisfies
\begin{itemize}
\item
$\tau_{x}f(y)=\tau_{y}f(x),$

\smallskip
\item
$\displaystyle \int_{\mathbb{R}^N}\tau_{x}f(y)g(y)\, d\omega(y)
=\int_{\mathbb{R}^N}f(y)\tau_{-x}g(y)  \, d\omega(y)$.
\end{itemize}

The following specific formula was obtained by R\"osler \cite{Roesler2003}
for the Dunkl translations of radial functions $f(x)=\tilde{f}({\|x\|})$\,:
\begin{equation}\label{translation-radial}
\tau_{x}f(-y)=\int_{\mathbb{R}^N}{\big(\tilde{f}\circ A\big)}(x,y,\eta)\,d\mu_{x}(\eta),{\qquad\forall\;x,y\in\mathbb{R}^N.}
\end{equation}
Here
\begin{equation*}
A(x,y,\eta)=\sqrt{{\|}x{\|}^2+{\|}y{\|}^2-2\langle y,\eta\rangle}=\sqrt{{\|}x{\|}^2-{\|}\eta{\|}^2+{\|}y-\eta{\|}^2}
\end{equation*}
and $\mu_{x}$ is the probability measure occurring in \eqref{Rintegral}, which is supported in $\operatorname{conv}\mathcal{O}(x)$.

\smallskip

{\bf Heat kernel.}
Set $T_j=T_{e_j}$, where $\{e_1,\dots,e_N\}$ is the canonical basis of $\mathbb{R}^N$. Then, the Dunkl Laplacian ${\Delta}:=\sum_{j=1}^NT_{j}^2$ associated with $R$ and $\kappa$ is the differential-difference operator, which acts on $C^2$ functions by
$$
{\Delta}f(x)
{=\Delta_{\text{eucl}}f(x)
+\sum_{\alpha\in R}\kappa(\alpha)\delta_\alpha f(x)}
=\Delta_{\text{eucl}}f(x)
+2\sum_{\alpha\in R^+}  \kappa(\alpha)\delta_\alpha f(x),
$$
where
$$
\delta_\alpha f(x)
=\frac{\partial_\alpha f(x)}{\langle\alpha,x\rangle}-\frac{f(x)-f(\sigma_\alpha(x))}{\langle \alpha,x\rangle^2}
$$
and $\Delta_{\text{eucl}}=\sum_{j=1}^N\partial_{j}^2$ is the classic Laplacian on $\mathbb{R}^N$.
In particular, we have
\begin{equation}\label{FDelta}
\mathcal{F}_{\kappa}({\Delta}f)(\xi)=-\|\xi\|^2\mathcal{F}_{\kappa} f(\xi)
\end{equation}
and
\begin{equation}\label{FTj}
\mathcal{F}_{\kappa}({T_j}f)(\xi)=i\xi_{j}\mathcal{F}_{\kappa}f(\xi).
\end{equation}

The operator $\Delta$ is essentially self-adjoint on $L^2(\mathbb{R}^N, d\omega)$
 (see for instance \cite[Theorem\;3.1]{AH})
and generates the heat semigroup
\begin{equation}\label{heat_semigroup}
H_tf(x)=e^{t{\Delta}}f(x)=\int_{\mathbb{R}^N}
h_t(x,y)f(y)\,   d\omega (y).
\end{equation}
Here the heat kernel $h_t(x,y)$ is a $C^\infty$ function in all variables $t>0$, $x,y\in\mathbb{R}^N$, which satisfies
$$
h_t(x,y)=h_t(y,x)
{>0\quad\text{and}\quad}
\int_{\mathbb{R}^N} h_t(x,y)\,d\omega(y)=1.
$$

Specifically, for every $t>0$ and for every $x,y\in\mathbb{R}^N$,
\begin{equation}\label{Expression1HeatKernel}
h_t(x,y)
=\tau_{x}h_t(-y),
\end{equation}
where
\begin{equation*}
h_t(x)
=c_{\kappa}^{-1}\,(2t)^{-\frac{\mathbf{N}}{2}}\,\exp\left(-\frac{{\|x\|}^2}{4t}\right).
\end{equation*}
Note that we can write the fractional Riesz transform as
\begin{align}\label{eqFRT}
R_j^{\alpha} f=-T_j(-\Delta)^{-\frac{1+\alpha}{2}} f=-C_{\alpha} \int_0^{\infty} T_j e^{t \Delta} f t^{\frac{1+\alpha}{2}-1}d t,
\end{align}
where $0<\alpha<N$.
In \cite[Lemma 3.3]{DA2020}, for all $x, y \in \mathbb{R}^N$ and $t>0$,
\begin{equation}\label{eqn:TjHeat}
T_j h_t(x, y)=\frac{y_j-x_j}{2 t} h_t(x, y).
\end{equation}
Here are some useful estimates of heat kernels.
\begin{lemma}\label{theoremGauss}
{\rm(a) } {\rm{(}}\cite{DA2020}{\rm{)}}
There are constants $C,c>0$ such that
\begin{equation}\label{Gauss}
 h_t(x,y) \leq C
\left(1+\frac{\|x-y\|}{\sqrt{t}}\right)^{-2}
\frac{1}{V(x,y,\sqrt{t\,})}\exp\left(-c  \frac{d(x,y)^2}{ t}\right)
\end{equation}
{for every $t>0$ and for every \,$x,y\in\mathbb{R}^N$.}
\par\noindent
{\rm(b) }{\rm{(}}\cite{DA2020}{\rm{)}}
There are constants $C,c>0$ such that
\begin{equation}\label{Holder}
\left|h_t(x,y)-h_t(x,y')\right|\leq C\frac{{\|y-y'\|}}{\sqrt{t}}\left(1+\frac{\|x-y\|}{\sqrt{t}}\right)^{-2}
\frac{1}{V(x,y,\sqrt{t\,})}
\,\exp\left(-c  \frac{d(x,y)^2}{ t}\right)
\end{equation}
for every $t>0$ and for every $x,y,y'\in\mathbb{R}^N$ such that $\|y-y'\|<\sqrt t$.
\par\noindent
{\rm(c) } {\rm{(}}\cite{ADH}{\rm{)}}
There exist positive constants $C$ and $c$ such that
\begin{equation}\label{gaussian_lower}
h_t(x,y)\geq\frac{C}{\min\left\{\omega(B(x,\sqrt{t})),  \omega(B(y,\sqrt{t}))\right\}}
\exp\left(-c \frac{\|x-y\|^{2}}  { t}\right)
\end{equation}
for every $t>0$ and for every $x,y\in\mathbb{R}^N$.
\end{lemma}











\smallskip

{\bf VMO spaces.}
Here, we give the definitions of the $\operatorname{CVMO}_{\text {Dunkl}}$ and $\mathrm{VMO}_d$ spaces associated with the Euclidean metric and the Dunkl metric. Let $r_B$ be the radius of the Euclidean ball $B \subseteq \mathbb{R}^N$. First, we define the central VMO space in the Dunkl setting as follows:
$$
\operatorname{CVMO}_{\text {Dunkl }}(\mathbb{R}^N)=\left\{\, b \in \mathrm{CBMO}_{\text {Dunkl }}(\mathbb{R}^N):\eqref{eq(1)}-\eqref{eq(3)} \text { hold}\,\right\}
$$
where
\begin{equation}\label{eq(1)}
\lim _{r_B \rightarrow 0} \sup _{ 0\in B} \frac{1}{\omega(B)} \int_{B}|b(x)-b_{B}| \,d \omega(x)=0,
\end{equation}
\begin{equation}\label{eq(2)}
 \lim _{r_B \rightarrow \infty} \sup _{  0\in B } \frac{1}{\omega(B)} \int_{B}|b(x)-b_{B}|\, d \omega(x)=0,
 \end{equation}
\begin{equation}\label{eq(3)}
 \lim _{r \rightarrow \infty} \sup _{B \subseteq \mathbb{R}^N, B \cap B(0, r)=\emptyset} \frac{1}{\omega(B)} \int_B|b(x)-b_B| \,d \omega(x)=0.
 \end{equation}

Next, we define the VMO space associated with the Dunkl metric as follows:
$$
\mathrm{VMO}_d(\mathbb{R}^N)=\left\{\, b \in \mathrm{BMO}_d(\mathbb{R}^N):\eqref{eq(4)}-\eqref{eq(6)} \text { hold}\,\right\}
$$
where
\begin{equation}\label{eq(4)}
\lim _{r_B \rightarrow 0} \sup _{\mathcal{O}(B) \subseteq \mathbb{R}^N} \frac{1}{\omega(\mathcal{O}(B))} \int_{\mathcal{O}(B)}|b(x)-b_{\mathcal{O}(B)}| \, d \omega(x)=0,
\end{equation}
\begin{equation}\label{eq(5)}
\lim _{r_B \rightarrow \infty} \sup _{\mathcal{O}(B) \subseteq \mathbb{R}^N} \frac{1}{\omega(\mathcal{O}(B))} \int_{\mathcal{O}(B)}|b(x)-b_{\mathcal{O}(B)}| \,   d \omega(x)=0,
\end{equation}
\begin{equation}\label{eq(6)}
 \lim _{r \rightarrow \infty} \sup _{B \subseteq \mathbb{R}^N, \, \mathcal{O}(B) \cap B(0, r)=\emptyset} \frac{1}{\omega(\mathcal{O}(B))} \int_{\mathcal{O}(B)}|b(x)-b_{\mathcal{O}(B)}|  \,   d \omega(x)=0.
\end{equation}

\smallskip

{\bf Maximal function.}
The Hardy-Littlewood maximal function $M$ in the Dunkl setting  is defined as
$$
M f(x)=\sup _{ x\in B} \frac{1}{\omega(B)} \int_B|f(y)| \,   d \omega(y)
$$
and the fractional maximal function $M^{\beta} $ is defined as
$$
M^{\beta} f(x)=\sup_{ x\in B} \frac{1}{\omega(B)^{1-\beta/\mathbf{N}}} \int_{B}|f(y)| \,  d \omega(y),
$$
for any $0<\beta <\mathbf{N}$.

Note that $(\mathbb{R}^N, \|\cdot\|, d\omega)$ is a
space of homogeneous type in the sense of Coifman and Weiss.
Then, we have $M$ is bounded on $L^p(\mathbb{R}^N, \|\cdot\|, d\omega)$ (see \cite{St2})
and $M^{\beta}$ is bounded from $L^p(\mathbb{R}^N,\|\cdot\|, d\omega)$
to $L^q(\mathbb{R}^N, \|\cdot\|,d\omega)$
with $\displaystyle \frac{1}{q}=\frac{1}{p}-\frac{\beta}{\mathbf{N}}$ (see \cite{BS1994}).
Moreover, the sharp maximal function $f^{\sharp}$ is defined as
$$
f^{\sharp}(x)=\sup _{x \in B} \frac{1}{\omega(B)} \int_B|f(y)-f_B|  \,  d \omega(y).
$$
From \cite[Theorem 5.5]{HT2019} and \cite[Theorem 3.1]{HLLW}, we have the following lemma.

\begin{lemma}{\rm{(}}\cite{HT2019, HLLW}{\rm{)}}\label{bmoApprox}
 Let $1 \leq p<\infty$ and $f$ is measurable function on $\mathbb{R}^N$. Then
$$
\|b\|_{\mathrm{BMO}_d} \approx \sup _{B \subseteq \mathbb{R}^N}
\left(\frac{1}{\omega(\mathcal{O}(B))} \int_{\mathcal{O}(B)} \left|b(x)-b_{\mathcal{O}(B)}\right|^p
d \omega(x)\right)^{ 1 / p}.
$$
\end{lemma}

\medskip

\section{Proof of Theorem \ref{operatorbound}}\label{proofoperator}
First, we will give a lemma to provide the size condition and smoothness condition for the kernel of the fractional Dunkl Riesz transform.

\begin{lemma}\label{smoothness}
 For $0<\alpha<N$, there exists a constant $C$ such that for $j\in \left\{ 1,2, \ldots, N\right\}$ and for every $x, y$ with $d(x, y) \neq 0$,
\begin{equation}\label{sizecondi}
 \left|R_j^{\alpha}(x, y)\right| \leq C  \frac{d(x,y)^{\alpha}}{\omega(B(x, d(x, y)))},
\end{equation}
\begin{equation}\label{smoothcondi1}
\left|R_j^{\alpha}(x, y)-R_j^{\alpha}(x, y^{\prime})\right| \leq C \frac{\|y-y^{\prime}\|}{\|x-y\|}
 \frac{d(x,y)^{\alpha}}{\omega(B(x, d(x, y)))}
\end{equation}
for $\|y-y^{\prime}\| \leq d(x, y) / 2 $, and
\begin{equation}\label{smoothcondi2}
 \left|R_j^{\alpha}(x^{\prime}, y)-R_j^{\alpha}(x, y)\right| \leq C \frac{\|x-x^{\prime}\|}{\|x-y\|}
\frac{d(x,y)^{\alpha}}{\omega(B(x, d(x, y)))}
\end{equation}
for $\|x-x^{\prime}\| \leq d(x, y) / 2$.
\end{lemma}
\begin{remark}
It is important to emphasize that for $0<\alpha <N-1$,
we have the estimate:
\begin{equation*}
 \left|R_j^{\alpha}(x, y)\right| \leq C \frac{d(x,y)}{\|x-y\|} \frac{d(x,y)^{\alpha}}{\omega(B(x, d(x, y)))},
\end{equation*}
which recovers the corresponding size estimate for the kernel of the Riesz transforms established in \cite[Theorem 1.1]{HLLW} as $\alpha =0$.

When $N-1<\alpha<N$, we can only obtain the size estimate shown in \eqref{sizecondi}. However, this slight difference will not affect the boundedness of $R_j^{\alpha}$ and their commutators.
\end{remark}

\begin{proof}[Proof of Lemma \ref{smoothness}]
Since $R_j^\alpha$ is a convolution operator with the kernel $R_j^{\alpha}(x,y)$, then
\begin{align*}
 R_j^\alpha f(x)
 = \int_{\mathbb{R}^N} R_j^\alpha(x, y) f(y) \,d\omega( y)    .
\end{align*}
Combining  \eqref{eqFRT} and \eqref{eqn:TjHeat}, we have \begin{align*}
 R_j^\alpha f(x)=&\left(-C_\alpha  \int_0^\infty T_j e^{t \Delta} t^{\frac{1+\alpha}{2}-1}d t\right) f(x)\\
 =& -C_\alpha \int_0^{\infty} T_j \int_{\mathbb{R}^N} h_t(x, y) f(y) \,d\omega( y)\, t^{\frac{1+\alpha}{2}-1} d t    \\
 =& -C_\alpha \int_0^{\infty} \int_{\mathbb{R}^N} \frac{y_j-x_j}{2 t}
 h_t(x, y) f(y) \,d \omega(y) \,t^{\frac{1+\alpha}{2}-1} d t    \\
 =& \int_{\mathbb{R}^N}\left(-C_\alpha \int_0^{\infty} \frac{y_j-x_j}{2 t} h_t(x, y)  t^{\frac{\alpha-1}{2}} d t\right) f(y) \,d\omega( y)    .
\end{align*}
Therefore,
\begin{align}\label{kerneldefi}
 R_j^\alpha(x, y)=-\frac{C_\alpha}{2}\left(y_j-x_j\right) \int_0^\infty
 h_t(x, y) t^{\frac{\alpha-1}{2}-1} d t.
\end{align}
Then, by Lemma \ref{theoremGauss}, we have the estimate
\begin{align*}
 &\left|R_j^\alpha(x, y)\right| \\
 \lesssim &\left\{\int_0^{d(x,y)^2} +\int_{d(x,y)^2}^{\|x-y\|^2}+\int_{\|x-y\|^2}^\infty  \right\}\frac{|y_j-x_j| }{V(x, y, \sqrt{t})}\left(1+\frac{\| x -y \|}{\sqrt{t}}\right)^{-2} \exp\left(-c \frac{d(x, y )^2}{t}\right) t^{\frac{\alpha-1}{2}-1} d t \\
 =:& R_{I}(x,y)+R_{II}(x,y)+R_{III}(x,y)
\end{align*}
When $t<d(x, y)^2$, then $\displaystyle t^{\frac{\alpha}{2}}<d(x, y)^{\alpha}$. Applying the second inequality in \eqref{growth}, we have
\begin{align*}
R_{I}(x, y)
\lesssim &\frac{|y_j-x_j|}{\omega(B(x, d(x,y))) } \int_0^{d(x, y)^2}
\left(\frac{d (x,y)}{\sqrt{t}}\right)^{\mathbf{N}}
\left(\frac{\| x -y \|}{\sqrt{t}}\right)^{-2} \exp\left(-c \frac{d(x, y )^2}{t}\right) t^{\frac{\alpha-1}{2}-1} d t\\
\lesssim & \frac{|y_j-x_j|}{\omega(B(x, d (x, y)))} \int_0^{d(x, y)^2}
\frac{d(x, y)^\alpha}{\|x-y\|^2} \left(\frac{d(x, y)}{\sqrt{t}}\right)^{\mathbf{N}} \left(\frac{t}{d(x, y)^2}\right)^{\frac{\mathbf{N}+1}{2}} t^{-\frac{1}{2}} d t \\
\lesssim & \frac{|y_j-x_j|}{\|x-y\|^2} \frac{d(x, y)^\alpha}{\omega(B(x, d(x, y)))} \int_0^{d(x, y)^2}d(x, y)^{-1} d t \\
\lesssim & \frac{|y_j-x_j|}{\|x-y\|}  \frac{d(x, y)}{\|x-y\|}
\frac{d(x, y)^\alpha}{\omega(B(x, d(x, y)))} .
\end{align*}

When $d(x, y)^2\leq t\leq \|x-y\|^2$, then $\displaystyle t^{\frac{-N+\alpha}{2}}\leq d(x, y)^{-N+\alpha}$. By  the first inequality in \eqref{growth},
\begin{align*}
R_{II}(x,y)
\lesssim &\frac{|y_j-x_j|}{\omega(B(x, d(x,y))) }
 \int_{d(x, y)^2}^{\|x-y\|^2}
 \left(\frac{d (x,y)}{\sqrt{t}}\right)^{N}
\left(\frac{\| x -y \|}{\sqrt{t}}\right)^{-2} t^{\frac{\alpha-1}{2}-1} d t\\
= & \frac{|y_j-x_j|}{\|x-y\|^2}
\frac{d(x, y)^{N}}{\omega(B(x, d(x, y)))}
   \int_{d(x, y)^2}^{\|x-y\|^2} t^{\frac{-N+\alpha}{2}} t^{-\frac{1}{2}} d t  \\
   \lesssim & \frac{|y_j-x_j|}{\|x-y\|^2}
\frac{d(x, y)^{\alpha}}{\omega(B(x, d(x, y)))}
   \int_{d(x, y)^2}^{\|x-y\|^2} t^{-\frac{1}{2}} d t  \\
\lesssim & \frac{|y_j-x_j|}{\|x-y\|}
\frac{d(x, y)^{\alpha}}{\omega(B(x, d(x, y)))}.
\end{align*}

When $t>\|x-y\|^2 $, then
$\displaystyle t^{\frac{-N+\alpha}{2}}<d(x,y)^{-N+\alpha}$ due to $0<\alpha<N$. Similarly,
\begin{align*}
 R_{III}(x,y)
\lesssim & | y_j-x_j| \int_{\|x-y\|^2}^{\infty} \frac{1}{\omega(B(x, \sqrt{t}))} t^{\frac{\alpha-1}{2}-1} d t \\
\lesssim & \frac{|y_j-x_j|}{\omega(B(x, d(x, y)))} \int_{\|x-y\|^2}^{\infty}\left(\frac{d(x, y)}{\sqrt{t}}\right)^{N} t^{\frac{\alpha-1}{2}-1} d t \\
= & \frac{|y_j-x_j|}{\omega(B(x, d (x, y)))} d(x, y)^N \int_{\|x-y\|^2}^{\infty} t^{\frac{-N+\alpha}{2}} t^{-\frac{1}{2}-1} d t  \\
\lesssim & \frac{|y_j-x_j|}{\omega(B(x, d(x, y)))} d(x, y)^{N} d(x, y)^{-N+\alpha}\int_{\|x-y\|^2}^{\infty}  t^{-\frac{1}{2}-1} d t \\
\lesssim & \frac{|y_j-x_j|}{\|x-y\|}\frac{d(x, y)^{\alpha}}{\omega(B(x, d(x, y)))}.
\end{align*}

Hence, we obtain \eqref{sizecondi}.

It remains to prove \eqref{smoothcondi1}, noting \eqref{smoothcondi2} can be obtained similarly. From \eqref{kerneldefi} and the non-negativity of heat kernels, we have
\begin{align*}
 \left|R_j^{\alpha}(x, y)-R_j^{\alpha}(x, y^{\prime})\right|
\leq & |y_j-y_j^{\prime}| \int_0^{\infty}h_t(x, y) t^{\frac{\alpha-1}{2}-1} d t \\
& +|y_j^{\prime}-x_j| \int_0^{\infty}|h_t(x, y)-h_t(x, y^{\prime})| t^{\frac{\alpha-1}{2}-1} d t\\
=: & \mathcal R_{I}(x,y,y') +\mathcal R_{II}(x,y,y')
\end{align*}

For $ \mathcal R_{I}(x,y,y') $. By the proof of \eqref{sizecondi},
\begin{align*}
 \mathcal R_{I}(x,y,y') \lesssim &  |y_j-y_j^{\prime} | \int_0^{\infty}
\frac{1}{V(x, y, \sqrt{t})}
\left(1+ \frac{\|x-y\|}{\sqrt{t}}\right)^{-2} \exp\left(-c\frac{d(x, y)^2}{t}\right) t^{\frac{\alpha-1}{2}-1} d t \\
\lesssim &  \frac{|y_j-y_j^{\prime}|}{\|x-y\|}
\frac{d(x,y )^\alpha}{\omega(B(x, d (x, y)))} .
\end{align*}

For $\mathcal R_{II}(x,y,y')$. We split
\begin{align*}
&\mathcal R_{II}(x,y,y') \\
= \,& |y_j^{\prime}-x_j| \left\{ \int_0^{\|y-y'\|^2} +\int_{\|y-y'\|^2}^{d(x, y)^2} +\int_{d(x, y)^2}^{\infty} \right\}\left|h_t(x, y)-h_t(x, y^{\prime})\right| t^{\frac{\alpha-1}{2}-1}d t \\
=: \,& \mathcal R_{II}^{(1)}(x,y,y')+ \mathcal R_{II}^{(2)}(x,y,y')+\mathcal R_{II}^{(3)}(x,y,y').
\end{align*}

 Since $\|y-y^{\prime}\| \leq d(x,y) / 2 $,
 we have $ d(x, y) \approx d(x, y^{\prime})$ and $\|x-y\| \approx\|x-y' \|$.
Note that
$$
|y_j^{\prime}-x_j| \leq \|y^{\prime}-x\| \leq \frac{3}{2}\|x-y\|.
$$

When $ t<\|y-y'\|^2$, then
 \begin{align*}
\mathcal R_{II}^{(1)}(x,y,y') \leq & \frac{3}{2}\|x-y\|  \int_0^{d(x, y)^2}
  \frac{\|y-y^{\prime} \|}{\sqrt{t}} \left  (|h_t( x, y)|+|h_t( x, y')| \right) t^{\frac{\alpha-1}{2}-1} d t \\
 \leq & \frac{3}{2}\|x-y\|\,\|y-y^{\prime}\| \int_0^{d(x, y)^2} \frac{1}{\sqrt{t}} \left\{\frac{1}{\omega(x, \sqrt{t})} \frac{t}{\|x-y\|^2}
  \exp\left(-c \frac{d (x, y)^2}{t}\right)\right. \\
&\hspace{4cm} \left.+\frac{1}{\omega(x, \sqrt{t})} \frac{t}{\|x-y'\|^2}
  \exp\left(-c \frac{d (x, y')^2}{t}\right) \right\}t^{\frac{\alpha-1}{2}-1} d t \\
 \lesssim & \frac{\|x-y\| \,  \|y-y^{\prime}\|}{\omega(B(x, d (x, y)))} \int_0^{d (x, y)^2}
  \left(\frac{d(x, y)}{\sqrt{t}}\right)^{\mathbf{N}}
  \frac{t}{\|x-y\|^2} \exp\left(-c\frac{d(x, y)^2}{t}\right) t^{-\frac{1}{2}} t^{\frac{\alpha-1}{2}-1} d t \\
\lesssim &\frac{\|x-y\|\,  \|y-y '\|}{\omega(B(x, d (x, y)))}
\frac{d(x, y)^{\alpha}}{\|x-y\|^2}
\int_0^{d(x, y)^2}
\left(\frac{d( x, y)}{\sqrt{t}}\right)^{\mathbf{N}}
\left(\frac{t}{d(x, y)^2}\right)^{\frac{\mathbf{N}+2}{2}}   \frac{dt}{t} \\
\lesssim & \frac{\|y-y'\|}{\|x-y\|} \frac{d(x, y)^{\alpha}}{\omega(B(x, d (x, y)))}.
\end{align*}

When $\|y-y^{\prime}\|^2\leq t\leq d(x,y)^2$, then by \eqref{Holder},
\begin{align*}
\mathcal R_{II}^{(2)}(x,y,y') \leq & \frac{3}{2}\|x-y\| \int_{\|y-y^{\prime}\|^2}^{d(x, y)^2} \frac{\|y-y^{\prime}\|}{\sqrt{t}}
\left(1+ \frac{\|x-y\|}{\sqrt{t}}\right)^{-2}
\frac{1}{V(x, y, \sqrt{t})} \exp\left(-c \frac{d (x, y)^2}{t}\right)
 t^{\frac{\alpha-1}{2}-1}d t \\
 \leq & \frac{3}{2} \frac{\|x-y\|\, \|y-y^{\prime}\| \,d (x, y)^{\alpha}}{\|x-y\|^2  \omega(B(x, d (x, y)))}
 \int_{\|y-y^{\prime}\|^2}^{d(x, y)^2}
 \left(\frac{d(x, y)}{\sqrt{t}}\right)^{\mathbf{N}}
 \left(\frac{t}{d (x, y)^2}\right)^{\frac{\mathbf{N}+2}{2}} \frac{dt}{t}  \\
\lesssim & \frac{\|y-y^{\prime}\|}{\|x-y\|}
\frac{d(x, y)^\alpha}{\omega(B(x, d (x, y)))} .
\end{align*}

When $ t> d(x,y)^2$, $\|y-y'\|<d(x,y)/2<\sqrt{t}/2$.
By \eqref{Holder} again,
\begin{align*}
\mathcal R_{II}^{(3)}(x,y,y')\leq  &\frac{3}{2}\|x-y\| \int_{d(x, y)^2}^{\infty}  \left |h_t(x, y)-h_t(x, y^{\prime})\right | t^{\frac{\alpha-1}{2}-1}d t\\
\lesssim & \frac{3}{2}\|x-y\|\int_{d(x, y)^2}^{\infty}
 \frac{\|y-y^{\prime}\|}{\sqrt{t}}
\left(1+ \frac{\|x-y\|}{\sqrt{t}}\right)^{-2}
\frac{1}{V(x, y, \sqrt{t})} \exp\left(-c \frac{d (x, y)^2}{t}\right)
 t^{\frac{\alpha-1}{2}-1}d t \\
 \leq & \frac{3}{2} \frac{\|x-y\|  \, \|y-y^{\prime}\| }{\|x-y\|^2}
 \frac{d (x, y)^{N}}{\omega(B(x, d (x, y)))}
 \int_{d(x, y)^2}^{\infty}
 t^{-\frac{N}{2}+\frac{\alpha}{2}-1}d t \\
 \lesssim &  \frac{\|y-y^{\prime}\| }{\|x-y\|}
 \frac{d (x, y)^{\alpha}}{\omega(B(x, d (x, y)))} .
 \end{align*}
We complete the proof.
\end{proof}

\smallskip

\begin{proof}[Proof of Theorem \ref{operatorbound}]
By \eqref{sizecondi}, we have
\begin{align*}
 \left|R_j^\alpha f(x)\right|
 \leq & \int_{d(x, y)<R} \frac{d (x, y)^\alpha}{\omega(B(x, d (x, y)))}|f(y)| d \omega(y)  +\int_{d(x, y) \geq R} \frac{d (x, y)^\alpha}{\omega(B(x, d (x, y)))}|f(y)| d \omega(y) \\
=:& \mathscr R_I(f) (x)+ \mathscr R_{II}(f) (x).
\end{align*}
Since for any $\lambda R>d(x,y)$,
 $$\frac{1}{\omega(B(x, d (x, y)))} \leq
 \frac{1}{\omega(B(x, \lambda R))}
 \left(\frac{\lambda R}{d (x, y)}\right)^{\mathbf{N}}. $$
 Then, we have
\begin{align*}
\mathscr R_I(f) (x) \leq & \sum_{i={-\infty}}^0
 \int_{2^{i-1} R \leq d(x, y)<2^i R}
 \frac{d(x, y)^\alpha}{\omega(B(x, 2^{i} R))}
 \left(\frac{2^i R}{d(x, y)}\right)^{\mathbf{N}}|f(y)| d \omega(y)  \\
\leq & \sum_{i={-\infty}}^0
\frac{(2^i R)^\alpha}{\omega(B(x, 2^{i} R))}
 \int_{ d(x, y)<2^i R}
 \left(\frac{2^i R}{2^{i-1} R}\right)^{\mathbf{N}}|f(y)| d \omega(y)  \\
 \leq &\sum_{i={-\infty}}^0(2^i R)^\alpha
 \frac{|G|}{\omega(\mathcal{O}(B(x, 2^i R)))}
 \int_{\mathcal{O}(B(x, 2^i R))}|f(y)| d \omega(y) \\
 \leq &\sum_{i={-\infty}}^0(2^i R)^\alpha
 \frac{|G|}{\omega(\bigcup_{\sigma \in G}B(\sigma(x),2^i R))}
 \sum_{\sigma \in G}\int_{B(\sigma(x),2^i R)}|f(y)| d \omega(y) \\
 \leq &|G|\sum_{i={-\infty}}^0(2^i R)^\alpha\sum_{\sigma \in G}
 \frac{1}{\omega(B(\sigma(x),2^i R))}
\int_{B(\sigma(x),2^i R)}|f(y)| d \omega(y) \\
 \lesssim & R^\alpha \sum_{\sigma \in G}M f(\sigma(x)),
\end{align*}
and
\begin{align*}
\mathscr R_{II}(f) (x) \leq & \sum_{i = 0}^\infty \int_{2^i R \leq d(x, y) < 2^{i+1} R}
 \frac{d(x, y)^\alpha}{\omega(B(x, 2^{i+1}R))}
 \left(\frac{2^{i+1} R}{d (x, y)}\right)^{\mathbf{N}}|f(y)| d \omega ( y) \\
 \lesssim & \sum_{i=0}^\infty
 (2^{i+1} R)^{\alpha}|G|^{1/p'}
 \omega(B(x, 2^{i+1}R))^{-1+1/p'}\|f\|_p .
\end{align*}
From \eqref{ballmeasuresim}, we have
$$
\omega(B(x, 2^{i+1} R)) \geq (2^{i+1} R)^{\mathbf{N }},
$$
thus
$$
\mathscr R_{II}(f) (x)
\lesssim  \sum_{i = 0}^\infty
 (2^i R)^{-\mathbf{N}/q}\|f\|_p
\lesssim  R^{-\mathbf{N}/q}\|f\|_p.
$$

Set
$$
R^\alpha \sum_{\sigma \in G}M f(\sigma(x)) =R^{-\mathbf{N}/q}\|f\|_p,
$$
that is, we take
$$ R=\|f\|_p^{p/\mathbf{N}} \Big (\sum_{\sigma \in G}M f(\sigma(x))\Big )^{-p/\mathbf{N}} . $$

Then
\begin{align}\label{improesti}
 \left|R_j^\alpha f(x)\right|
 \leq & \int_{\mathbb{R}^N} \frac{d (x, y)^\alpha}{\omega(B(x, d (x, y)))}|f(y)| d \omega(y)  \\\nonumber
\lesssim &\|f\|_p^{-p/q+1}
\Big (\sum_{\sigma \in G}M f(\sigma(x))\Big )^{p/q} .
\end{align}
This yields that
$$
 \|R_j^\alpha f\|_q
\lesssim \|f\|_p^{-p/q+1}\Big \| \Big (\sum_{\sigma \in G}M f(\sigma(\cdot))\Big )^{p/q} \Big \|_q
\lesssim \|f\|_p.
$$
We complete the proof of Theorem \ref{operatorbound}.\end{proof}

\medskip

\section{Proof of Theorem \ref{commutatorbound}}\label{proofcomubound}

\subsection{Upper bound of $[b,R_j^{\alpha}]$}

Suppose $b \in \mathrm{B M O}_d, 1<p<\infty$ and $f$ in
$L^p(\mathbb{R}^N, d\omega)$.
For any $x \in \mathbb{R}^N$ and for any ball $B=B\left(x_0, r\right) \subseteq \mathbb{R}^N$ containing $x$, we split $f=f_1+f_2$ with
$f_1=f \cdot  \mathsf{1}_{\mathcal{O}(5 B)}$. We have
\begin{align*}
 &[b, R_j^\alpha ](f)(y)\\
 =&b(y) R_j^\alpha (f) (y)-R_j^\alpha(b f)(y)  \\
 =&\left(b(y)-b_{\mathcal{O}(B)}\right) R_j^\alpha(f)(y)-R_j^\alpha \left((b(y)-b_{\mathcal{O}(B)}) f\right)(y) \\
= &\left(b(y)-b_{\mathcal{O}(B)}\right) R_j^\alpha(f)(y)
  -R_j^\alpha\left((b(y)-b_{\mathcal{O}(B)}) f_1\right)(y)
  -R_j^\alpha \left((b(y)-b_{\mathcal{O}(B)}) f_2\right)(y) \\
=: & \mathcal C_1(f)(y)+  \mathcal C_2(f)(y)+\mathcal C_3(f)(y).
\end{align*}
Then, we will consider the sharp maximal function of $[ b, R_{j}^\alpha ] f$.
\begin{align*}
\left([ b, R_j^\alpha] f\right)^{\#}(x)
=&\sup _{x \in B} \frac{1}{\omega(B)}
\int_B \left|[ b, R_j^\alpha] f ( y)-\big([ b, R_j^\alpha] f\big)_B \right| d \omega(y) \\
\leq & \sum_{i=1,2,3}\sup _{x \in B} \frac{1}{\omega(B)} \int_B \left|\mathcal C_i (f)(y)-(\mathcal C_i (f))_B\right| d \omega(y) .
\end{align*}
Then, choose $1<s<p$, we have the estimate
\begin{align*}
 &\frac{1}{\omega(B)} \int_B \left|\mathcal C_1(f)(y)-(\mathcal C_1(f))_B\right| d \omega(y) \\
\leq &  \frac{2}{\omega(B)} \int_B|\mathcal C_1(f)(y)| d \omega(y) \\
\lesssim &\frac{1}{\omega(B)} \int_B \left|b(y)-b_{\mathcal{O}(B)}\right|  \left|R_j^\alpha(f)(y)\right | d \omega(y) \\
\lesssim &
\left(\frac{1}{\omega(B)} \int_B \left|b(y)-b_{\mathcal{O}(B)}\right |^{s^{\prime}} d \omega(y)\right)^{1/s^{\prime}}
\left(\frac{1}{\omega(B)} \int_B  \left |R_j^\alpha(f)(y)\right |^s d \omega(y)\right)^{1/s} \\
\leq &
\left(\frac{|G|}{\omega(\mathcal{O}(B))} \int_{\mathcal{O}(B)} \left|b(y)-b_{\mathcal{O}(B)}\right|^s d \omega(y)\right)^{1 / s^{\prime}}
\left(M ( R_j^\alpha f)^s(x)\right)^{1 / s}\\
\lesssim  &
\|b\|_{\text {BMO}_d}\left(M[ ( R_j^\alpha f)^s](x)\right)^{1 / s}.
\end{align*}
By using \eqref{equival O vollum}, we are going to see
\begin{align*}
&\frac{1}{\omega(B)} \int_B \left|\mathcal C_2(f)(y)-(\mathcal C_2(f))_B\right | d \omega(y)\\
\lesssim &  \frac{1}{\omega(B)} \int_B \int_{\mathcal{O}(5B)}
\frac{d(y, z)^\alpha}{w(B(y, d(y, z))} \left|b(z)-b_{\mathcal{O}(B)}\right||f(z)| d \omega(z) d \omega(y)\\
 = &  \frac{1}{\omega(B)} \int_{\mathcal{O}(5B)} \left|b(z)-b_{\mathcal{O}(B)}\right| |f(z)|
\int_B\frac{d(y, z)^\alpha}{\omega(B(y, d(y, z))} d \omega(y)d \omega(z).
\end{align*}

Observe that for any $y \in B=B(x_0, r)=\{y:\|y-x_0\|<r\}$, we have
 $d(y, z)<6 r$ and then $y \in\{w: d(w, z)<6 r\}$, i.e.,
$B(x_0, r) \subseteq \mathcal{O}(B(z, 6 r))$. Thus, by \eqref{ballmeasuresim},
\begin{align*}
& \int_B \frac{d(y, z)^\alpha}{\omega(B (y, d(y, z)))} d \omega(y) \\
\leq & \int_{\mathcal{O}(B(z, 6 r))} \frac{d(y, z)^\alpha}{\omega(B(y, d(y , z))} d \omega(y) \\
 \leq & \sum_{i \leq 0} \int_{2^{i-1}\cdot 6r\leq d(y, z)<2^{i}\cdot 6r} \frac{d(y, z)^\alpha}{\omega(B(y, 2^{i}\cdot 6r))}\left(\frac{2^{i}\cdot 6r}{d(y, z)}\right)^{\mathbf{N}} d \omega(y) \\
\leq& \sum_{i\leq 0}(2^{i} \cdot 6 r)^\alpha \int_{d(y, z)<2^{i} \cdot 6 r} \frac{|G|}{\omega(\mathcal{O}(B(z, 2^{i} \times 6 r)))}\left(\frac{2^{i} \cdot 6 r}{2^{i-1} \cdot 6 r}\right)^{\mathbf{N}} d \omega(y) \\
\lesssim & \sum_{i\leq 0} 2^{i \alpha} r^\alpha
\lesssim  \omega(B(x_0, r))^{\alpha / \mathbf{N}}.
\end{align*}
 Then we can estimate
\begin{align*}
&\frac{1}{\omega(B)} \int_B \left|\mathcal C_2(f)(y)-(\mathcal C_2(f))_B\right | d \omega(y)\\
\leq & \frac{\omega(\mathcal{O}(5B))^{1 / s'}}{\omega(B)^{1-\alpha/\mathbf{N}}}
\left(\frac{1}{\omega(\mathcal{O}(5B))}
\int_{\mathcal{O}(5B)} |b(z)-b_{\mathcal{O}(B)}|^{s'} d \omega(z)\right)^{1 / s^{\prime}}
\left(\int_{\mathcal{O}(5 B)}|f(z)|^s d \omega(z)\right)^{1 / s}\\
\lesssim & \|b\|_{\mathrm{BMO}_d}
\left(\frac{1}{\omega(\mathcal{O}(5B))^{1-\alpha s / \mathbf{N}}}
\int_{\mathcal{O}(5B)} |f(z)|^s d \omega(z)\right)^{1 / s}\\
=& \|b\|_{\text {BMO}_d} \left(\frac{1}{\omega\big(\bigcup_{\sigma \in G}B(\sigma(x_0),5r)\big)^{1-\alpha s / \mathbf{N}}}
\int_{\bigcup_{\sigma \in G}B(\sigma(x_0),5r)}
 |f(z)|^s d \omega(z)\right)^{1 / s}\\
=&\|b\|_{\text {BMO}_d} \left(\sum_{\sigma \in G}\frac{1}{\omega(B(\sigma(x_0),5r))^{1-\alpha s / \mathbf{N}}}
\int_{B(\sigma(x_0),5r)}
 |f(z)|^s d \omega(z)\right)^{1 / s}\\
\leq & |G|^{1/s}
\|b\|_{\text {BMO}_d}
\left(M ^{\alpha s}( f)^s(x)\right)^{1 / s}.
\end{align*}

As for the last term, we have
\begin{align*}
&\frac{1}{\omega(B)} \int_B \left|\mathcal C_3(f)(y)-(\mathcal C_3(f))_B\right | d \omega(y)\\
\leq & \frac{2}{\omega(B)} \int_B  \left|\mathcal C_3(f)(y)-\mathcal C_3(f)(x_0)\right| d \omega(y)\\
\lesssim & \frac{1}{\omega(B)} \int_B \int_{\mathbb{R}^N \backslash \mathcal{O}(5 B)} \left|R_j^\alpha(\xi, y)-R_j^\alpha(\xi, x_0)\right||f(\xi)|
\left|b(\xi)-b_{ \mathcal{O}(B)}\right| d \omega(\xi) d \omega(y) \\
 \lesssim & \frac{1}{\omega(B)} \int_B \int_{\mathbb{R}^N \backslash \mathcal{O}(5 B)}
  \frac{\|y-x_0\|}{\|\xi-x_0\|}
  \frac{d(\xi, x_0)^\alpha}{\omega(B(\xi, d(\xi, x_0)))}
  |f(\xi)|\left |b(\xi)-b_{ \mathcal{O}(B)}\right | d \omega(\xi) d \omega(y)\\
 \lesssim & \frac{r}{\omega(B)} \int_B
 \left(\int_{\mathbb{R}^N \backslash \mathcal{O}(5 B)}
 \frac{1}{d(\xi, x_0) \omega(B(\xi, d(\xi, x_0)))}  \left|b(\xi)-b_{ \mathcal{O}(B)}\right|^{s^{\prime}} d \omega(\xi)\right)^{1 / s^{\prime}} \\
& \qquad \qquad \cdot \left(\int_{\mathbb{R}^N \backslash \mathcal{O}(5 B)} \frac{d(\xi, x_0)^{\alpha s}}{d(\xi, x_0) \omega(B(\xi, d(\xi, x_0)))}|f(\xi)|^s d \omega(\xi)\right)^{1 / s} d \omega(y).
\end{align*}
We have
\begin{align*}
&\int_{\mathbb{R}^N \backslash \mathcal{O}(5 B)} \frac{d(\xi, x_0)^{\alpha s}}{d(\xi, x_0) \omega(B(\xi, d(\xi, x_0)))}|f(\xi)|^s d \omega(\xi)\\
\leq &  \sum_{i \geq 0} \int_{2^{i}\cdot 5r\leq d(\xi, x_0)<2^{i+1}\cdot 5r} \frac{d(\xi, x_0)^{\alpha s}}{d(\xi, x_0) \omega(B(\xi, 2^{i+1}\cdot 5r))}\left(\frac{2^{i+1}\cdot 5r}{d(\xi, x_0)}\right)^{\mathbf{N}} |f(\xi)|^s d \omega(\xi) \\
\leq &  \sum_{i \geq 0} \int_{ d(\xi, x_0)<2^{i+1}\cdot 5r} \frac{|G|(2^{i+1}\cdot 5r)^{\alpha s}} {(2^{i}\cdot 5r) \omega(B(x_0, 2^{i+1}\cdot 5r))}\left(\frac{2^{i+1}\cdot 5r}{2^{i}\cdot 5r}\right)^{\mathbf{N}} |f(\xi)|^s d \omega(\xi) \\
\lesssim &  \sum_{i \geq 0} 2^{-i}r^{-1}\frac{\omega\big(\mathcal{O}(B(x_0, 2^{i+1}\cdot 5r))\big)^{\alpha s/ \mathbf{N}}} { \omega\big(\mathcal{O}(B(x_0, 2^{i+1}\cdot 5r))\big)}
\int_{ d(\xi, x_0)<2^{i+1}\cdot 5r} |f(\xi)|^s d \omega(\xi) \\
\lesssim & r^{-1} \sum_{i \geq 0} 2^{-i} M^{\alpha s}(f^s)(x)\\
\lesssim & r^{-1} M^{\alpha s}(f^s)(x).
\end{align*}
We also have
\begin{align*}
\int_{\mathbb{R}^N \backslash \mathcal{O}(5 B)}
 \frac{1}{d(\xi, x_0) \omega(B(\xi, d(\xi, x_0)))} \left|b(\xi)-b_{ \mathcal{O}(B)}\right|^{s^{\prime}} d \omega(\xi)
 \lesssim  r^{-1}\|b\|_{\mathrm{BMO}_d}^{s'}.
\end{align*}
Thus, we have
\begin{align*}
&\frac{1}{\omega(B)} \int_B \left|\mathcal C_3(f)(y)-(\mathcal C_3(f))_B\right | d \omega(y)\\
 \lesssim & \frac{r}{\omega(B)}
 \int_B \left(r^{-1}\|b\|_{\mathrm{BMO}_d}^{s'}\right)^{1 / s^{\prime}}
 \left(r^{-1} M^{\alpha s}(f^s)(x)\right)^{1 / s} d \omega(y)\\
 =&\|b\|_{\mathrm{BMO}_d} \left( M^{\alpha s}(f^s)(x)\right)^{1 / s}.
\end{align*}
Then, we conclude that
\begin{align*}
\left([ b, R_j^\alpha] f\right)^{\#}(x)
\lesssim \|b\|_{\mathrm{BMO}_d}\left(\left( M[(R_j^{\alpha}f)^s] (x)\right)^{1 / s}+\big ( M^{\alpha s}(f^s)(x)\big)^{1 / s}\right).
\end{align*}
Since $\displaystyle \frac{p}{s}>1$ and $\displaystyle\frac{s}{q}=\frac{s}{p}-\frac{\alpha s}{ \mathbf{N}}$, we have
$$
\left\|[ b, R_j^\alpha] f\right\|_q\leq \left\| \left([ b, R_j^\alpha] f\right)^{\#}\right\|_q
\lesssim \|b\|_{\mathrm{BMO}_d}\|f\|_p.
$$

\smallskip

\subsection{Lower bound of $[ b, R_j^\alpha]$} \label{Seclowerbound}
We first give the following lemma to provide an estimate for the
kernel of $R_j^\alpha$ in \eqref{kerneldefi}.
We borrow this idea from \cite[Theorem 1.2]{HLLW} and omit the proof here.
\begin{lemma}\label{kernellowerbound}
For $j=1,2, \ldots, N$ and for every ball $B=B(x_0, r) \subseteq \mathbb{R}^N$, there is another ball $\tilde{B}=B(y_0, r)$
such that $\|x_0-y_0\|=5 r$, and that for every $(x, y) \in B \times \tilde{B}$,
$$
\left|R_j^{\alpha}(x, y)\right| \gtrsim \frac{r^{\alpha}}{\omega(B(x_0, r))}.
$$
\end{lemma}
\begin{definition}
Let $b$ be finite almost everywhere on $\mathbb{R}^N$.
For $B \subseteq \mathbb{R}^N$ with $\omega(B)<\infty$,
we define a median value $m_b(B)$ of $b$ over $B$ to be a real number satisfying
$$
\omega\left(\left\{x \in B: b(x)>m_b(B)\right\}\right) \leq \frac{1}{2} \omega(B) \quad {\rm{ and }} \quad \omega\left(\left\{x \in B: b(x)<m_b(B)\right\}\right) \leq \frac{1}{2} \omega(B) .
$$
\end{definition}
For given $b \in L_{\operatorname{loc}}^1(\mathbb{R}^N, d \omega)$ and for any ball $B$, the oscillation $\Omega(b, B)$ is defined by
$$
\Omega(b, B):=\frac{1}{\omega(B)} \int_B|b(x)-b_B| d \omega(x).
$$
Let $B_0:=B(x_0,r)$ be any ball centred at $x_0$ with radius $r>0$ and containing the point $0$.
Then, we choose $\tilde{B}_0=B\left(\tilde{x}_0, r\right)$ with $\|\tilde{x}_0-x_0\|=5r$ such that $y_j-x_j \geq r$ and $\|x-y\| \approx r$ for $x \in B_0$ and $y \in \tilde{B}_0$.
From the expression of $R_j^{\alpha}$ in \eqref{kerneldefi}, it implies that $R_j^{\alpha}(x,y)$ does not change sign for any $(x,y)\in B_0 \times \tilde{B}_0$.

 Now, we choose two measurable sets
$$
E_1 \subseteq\left\{y \in \tilde{B}_0: b(y)<m_b(\tilde{B}_0)\right\} \quad \text { and } \quad E_2 \subseteq \left\{y \in \tilde{B}_0: b(y) \geq m_b(\tilde{B}_0)\right\}
$$
such that $\omega(E_i)=\frac{1}{2} \omega(\tilde{B}_0)$, $i=1,2$, and that $E_1 \cup E_2=\tilde{B}_0$, $E_1 \cap E_2=\emptyset$.
Moreover, we define
$$
B_1:=\left\{x \in B_0: b(x) \geq m_b(\tilde{B}_0)\right\} \quad \text { and } \quad B_2:=\left\{x \in B_0: b(x) \leq m_b(\tilde{B}_0)\right\}
$$

Now based on the definitions of $E_i$ and $B_i$, we have
\begin{align*}
& b(x)\geq m_b (\tilde{B}_0 ) > b(y), \quad(x, y) \in B_{1} \times E_{1} ;\\
& b(x)\leq m_b (\tilde{B}_0 ) \leq b(y), \quad(x, y) \in B_{2} \times E_{2}.
\end{align*}
Thus, for all $(x, y) \in B_i \times E_i, i=1,2$, we have that $b(x)-b(y)$ does not change sign and that
\begin{align*}
|b(x)-b(y)| & =\left|b(x)-m_b(\tilde{B}_0)+m_b(\tilde{B}_0)-b(y)\right| \\
& =\left|b(x)-m_b(\tilde{B}_0)\right|+\left|m_b(\tilde{B}_0)-b(y)\right| \geq  \left|b(x)-m_b(\tilde{B}_0)\right|.
\end{align*}
It is easy to check that
$$
\Omega(b, B_0)\leq
\frac{2}{\omega(B_0)} \int_{B_0} \left|b(x)-m_b(\tilde{B}_0)\right| d \omega(x).
$$
Let  $f_i=\mathsf 1_{E_i}$ for $i=1,2$. Since $\|\tilde{x}_0-x_0\|=5r$, we have $\omega(B_0)\approx \omega(\tilde{B}_0)$. By Lemma \ref{kernellowerbound}, we have
\begin{align*}
& \frac{r^{-\alpha}}{\omega(B_0)} \sum_{i=1}^2 \int_{B_0}  \left|[b, R_j^{\alpha}] f_i(x)\right| d \omega(x) \\
\geq &  \frac{r^{-\alpha}}{\omega(B_0)} \sum_{i=1}^2 \int_{B_i} \left|[b, R_j^{\alpha}] f_i(x)\right| d \omega(x)\\
=& \frac{r^{-\alpha}}{\omega(B_0)} \sum_{i=1}^2 \int_{B_i} \int_{E_i}|b(x)-b(y)| \left|R_j^{\alpha}(x, y)\right| d \omega(y) d \omega(x) \\
\gtrsim & \frac{r^{-\alpha}}{\omega(B_0)} \sum_{i=1}^2 \int_{B_i} \left|b(x)-m_b(\tilde{B}_0)\right| \frac{r^{\alpha}}{\omega(B_0)}
\int_{E_i} d \omega(y) d \omega(x) \\
\gtrsim & \frac{1}{\omega(B_0)} \sum_{i=1}^2 \int_{B_i} \left|b(x)-m_b(\tilde{B}_0)\right| d \omega(x) \\
\gtrsim &\left |\Omega(b, B_0)\right| .
\end{align*}
Next, from H\"{o}lder's inequality and the boundedness of $[b, R_j^\alpha]$, we deduce that
\begin{align*}
&\frac{r^{-\alpha}}{\omega(B_0)} \sum_{i=1}^2 \int_{B_0} \left|[b, R_j^{\alpha}] f_i(x)\right| d \omega(x)\\
\lesssim & \frac{r^{-\alpha}}{\omega(B_0)} \sum_{i=1}^2
\left(\int_{B_0}   \left|[b, R_j^{\alpha}] f_i(x)\right|^q d \omega(x)\right)^{1 / q} \omega(B_0)^{1 / q^{\prime}} \\
\lesssim & r^{-\alpha}\sum_{i=1}^2  \left\|[b, R_j^{\alpha}]\right\|_{L^p(\mathbb{R}^N, d\omega) \rightarrow L^q (\mathbb{R}^N, d\omega)} \omega(E_i)^{1 / p} \omega(B_0)^{-1 / q} .
\end{align*}
Since $0\in B_0$, then $\|x_0-0\|=\|x_0\|<r$. Note that $\big\|\frac{x_0}{\|x_0\|} -0\big\|=1$.
 By \eqref{growth}, the scaling property and \eqref{equival vollum}, we have
\begin{align*}
&\omega(B_0)=\omega(B(x_0,r))
\lesssim \omega \left(B(x_0,\|x_0\|)\right) \Big(\frac{r}{\|x_0\|}\Big)^{\mathbf{N}}\\
=&\omega\Big (B\Big (\frac{x_0}{\|x_0\|},1\Big )\Big )\, \|x_0\|^{\mathbf{N}}\Big(\frac{r}{\|x_0\|}\Big)^{\mathbf{N}}\approx
r^{\mathbf{N}}\omega(B(0,1)).
\end{align*}
Combining with $\omega(B_0)\approx \omega(\tilde{B}_0)\geq \omega(E_i)$, we have
\begin{align*}
|\Omega(b, B)|
\lesssim & r^{-\alpha}  \left\|[b, R_j^{\alpha}]\right\|_{L^p(\mathbb{R}^N, d\omega) \rightarrow L^q (\mathbb{R}^N, d\omega)} \omega(B_0)^{1 / p-1 / q}\\
\lesssim & \omega(B(0,1))^{\alpha / \mathbf{N}} \left\|[b, R_j^{\alpha}]\right\|_{L^p(\mathbb{R}^N, d\omega) \rightarrow L^q (\mathbb{R}^N, d\omega)}.
\end{align*}
The proof is complete.

\medskip

\section{Proof of Theorem \ref{commutatorcompact}}\label{proofcomucomp}
In this section, we prove the sufficiency of the compactness via adapting the idea from \cite{ChenDuongLi} via verifying the precompactness argument, that is, a version of Riesz--Kolmogorov theorem on space of homogeneous type. For necessity, we borrow the idea in \cite{LaceyLi} to the Dunkl setting.

It follows from \cite{ChenDuongLi} that the $\mathrm{VMO}_d(\mathbb{R}^N)$ are
equivalent to the closure of the set $\Lambda_{d, 0}(\mathbb{R}^N)$,
 the Lipschitz function space with the compact support, under the norm of the $\mathrm{BMO}_d$ space on the spaces of homogeneous
 type $(\mathbb{R}^N, d\omega)$.

\smallskip

\subsection{Sufficiency}
 By a density argument, to prove that when $b \in \mathrm{VMO}_d(\mathbb{R}^N)$, the commutator $[b, R_j^{\alpha}]$ is compact from $L^p(\mathbb{R}^N)$ to $L^q(\mathbb{R}^N)$, it suffices to show that $[b, R_j^{\alpha}]$ is compact for $b \in \Lambda_{d, 0}(\mathbb{R}^N)$.

A set $E$ is precompact if its closure is compact. Then, for $b \in \Lambda_{d, 0}(\mathbb{R}^N)$, to show $[b, R_j^{\alpha}]$ is compact from $L^p(\mathbb{R}^N)$ to $L^q(\mathbb{R}^N)$, it suffices to show that for every bounded subset $E \subseteq L^p(\mathbb{R}^N)$, the set $[b, R_j^{\alpha}] E$ is precompact on $L^q(\mathbb{R}^N)$.

Recall that the Riesz--Kolmogorov theorem (see for example \cite[Theorem 1]{GorkaMacios}) provides a common way to check precompactness. 
\begin{theorem} \label{RieszKolmogorov}{\rm{(}}\cite{GorkaMacios}{\rm{)}} Let $\mu$ be a doubling measure such that
$$
h(r):=\inf \big\{\mu(B(x, r)): x \in X\big\}>0 \quad \text { for each}\  r>0
$$
and assume $1<q<\infty$. Let $x_0 \in X$, then the subset $E$ of $L^q(X, \mu)$ is relatively compact if and only if the following conditions are satisfied:
\begin{itemize}
\item[$(a)$] $E$ is bounded;
\item[$(b)$]
$\displaystyle
\lim _{R \rightarrow \infty} \int_{X \backslash B(x_0, R)}|g(x)|^q d \mu(x)=0
$  uniformly for $g \in E$;

\item[$(c)$]
$\displaystyle
\lim _{r \rightarrow 0} \int_X|g(x)-g_{B(x, r)}|^q d \mu(x)=0
$ uniformly for  $g\in E$.
\end{itemize}
\end{theorem}
Now, we only need to show that $[b, R_j^{\alpha}] E$ satisfies conditions (a)--(c) of Theorem \ref{RieszKolmogorov}.
First, by Theorem \ref{commutatorbound} and the fact that $b \in \mathrm{BMO}_d(\mathbb{R}^N)$, it is direct to see that $[b, R_j^{\alpha}] E$ satisfies the condition (a).

Let's verify the condition (b). We may assume that $b \in \Lambda_{d, 0}(\mathbb{R}^N)$ with $\operatorname{supp} b \subseteq \mathcal{O}(B(x_0, R))$, $x_0 \in \mathbb{R}^N$. For $t>2$, set $K^c:=\{x \in \mathbb{R}^N: d(x, x_0)>t R\}$. Then we have
\begin{align*}
\left\|[b, R_j^{\alpha}] f\right\|_{L^q(\mathbb{R}^N \backslash \mathcal{O}(B(x_0, t R)), d \omega)}
\leq  \left\|b(R_j^{\alpha} f)\right\|_{L^q(K^c, d \omega)}
       +\left\|R_j^{\alpha}(b f)\right\|_{L^q(K^c, d \omega)}.
\end{align*}
For any $y\in K^c$, we have $d(y,x_0)>R$ and then $y \notin \mathcal{O}(B(x_0, R))$. Then we have
$$
 \left\|b(R_j^{\alpha} f)\right\|_{L^q(K^c, d \omega)}^q
 =\int_{d(y, x_0)>t R} \left|b(y) R_j^{\alpha}(f)(y)\right|^q d \omega(y)=0.
$$
By using \eqref{sizecondi} and the fact that if $d(x,y)\approx d(x,x_0)$, then $\omega(B(x, d(x , y))) \approx \omega(B(x_0, d(x , y))) \approx \omega(B(x_0, d(x , x_0)))$, we have
\begin{align*}
&\left\|R_j^{\alpha}(b f)\right\|_{L^q(K^c, d \omega)}^{q}\\
\leq & \int_{d(x, x_0)>t R}
\left(\int_{\mathcal{O}(B(x_0,  R))}\left|R_j^{\alpha}(x, y)\right||b(y)|\,|f(y)| d \omega(y)
\right)^q d \omega(x)\\
\lesssim & \int_{d(x, x_0)>t R}
\left(\int_{d(y, x_0)<R}\frac{d(x, y)^\alpha}{\omega(B(x, d(x , y)))}|b(y)|\, |f(y)| d \omega(y)
\right)^q d \omega(x)\\
\leq & \int_{d(x, x_0)>t R}\frac{d(x, x_0)^{\alpha q}}{\omega(B(x_0, d(x , x_0)))^q}
\left(\int_{d(y, x_0)<R}|b(y)|\, |f(y)| d \omega(y)
\right)^q d \omega(x)\\
\leq & \int_{d(x, x_0)>t R}\frac{d(x, x_0)^{\alpha q}}{\omega(B(x_0, d(x , x_0)))^q}
\left(\int_{d(y, x_0)<R}|b(y)|^{p'}d \omega(y)
\right)^{q/p'}
\|f\|_{p}^{q} d \omega(x)\\
\leq & \|f\|_{p}^{q} \,\|b\|_{\infty}^{q} \omega(\mathcal{O}(B(x_0, R)))^{q/p'}
\int_{d(x, x_0)>t R}\frac{d(x, x_0)^{\alpha q}}{\omega(B(x_0, d(x, x_0)))^q}
d \omega(x).
\end{align*}
Since
\begin{align*}
 &
\int_{d(x, x_0)>t R}\frac{d(x, x_0)^{\alpha q}}{\omega(B(x_0, d(x, x_0)))^{q}}
d \omega(x)\\
\leq & \sum_{i \geq 0}\int_{2^{i}tR\leq d(x, x_0)<2^{i+1}tR}
\frac{d(x, x_0)^{\alpha q}}{\omega\left(B(x_0, 2^{i+1}tR)\right)^q}
\left(\frac{2^{i+1}tR }{d( x_0,x)}\right)^{\mathbf{N}q}d \omega(x)\\
\lesssim & |G| \sum_{i \geq 0}
\left(2^{i+1}tR\right)^{\alpha q}\omega\left(B(x_0, 2^{i+1}tR)\right)^{-q+1}\\
\lesssim & \sum_{i \geq 0} 2^{i(\alpha q -\mathbf{N}(q-1))} t^{\alpha q -\mathbf{N}(q-1)} R^{\alpha q-\mathbf{N}(q-1)}\\
\lesssim & t^{-q\mathbf{N}/p'} R^{-q\mathbf{N}/p'}.
\end{align*}
Thus, we have
\begin{align*}
 \left\|[b, R_j^{\alpha}] f\right\|_{L^q((\mathcal{O}(B(x_0, t R)))^c, d \omega)}
  \lesssim  &  \left(\|f\|_{p}^{q}\, \|b\|_{\infty}^{q} \, \omega(\mathcal{O}(B(x_0, R)))^{q/p'} t^{-q\mathbf{N}/p'} R^{-q\mathbf{N}/p'} \right)^{1/q}\\
  \lesssim & \|f\|_{p} \,\|b\|_{\infty} \,\omega(\mathcal{O}(B(x_0, R)))^{1/p'} R^{-\mathbf{N}/p' }t^{-\mathbf{N}/p'},
\end{align*}
which tends to $0$, as  $t\rightarrow \infty$.

It remains to consider the condition (c).  Let $\varepsilon$ be a fixed constant in $(0,1/4)$. Then, we choose $r$ sufficiently small such that $r<\varepsilon^2$.
For the ball $ B(x,r)$, we have
\begin{align*}
 & \int_{\mathbb{R}^N}  \left|[b,R_j^{\alpha}]f(x)- \left([b,R_j^{\alpha}]f\right)_{B(x, r)}\right|^q d \omega(x)\\
 = &  \int_{\mathbb{R}^N}\bigg|\frac{1}{\omega(B(x,r))}\int_{B(x,r)}
  \left(   [b,R_j^{\alpha}]f(x)-[b,R_j^{\alpha}]f(z)\right) d\omega(z)\bigg|^q d \omega(x).
\end{align*}
For any $x\in \mathbb{R}^N$ and $z\in B(x,r)$, we split
\begin{align*}
&[b,R_j^{\alpha}]f(x)-[b,R_j^{\alpha}]f(z) \\
 = & \int_{\mathbb{R}^N} R_j^{\alpha}(x, y)\left(b(x)-b(y)\right) f(y) d \omega(y)
    -\int_{\mathbb{R}^N} R_j^{\alpha}(z, y)  \left(b(z)-b(y)\right) f(y) d \omega(y) \\
= & \int_{d(x, y)>\varepsilon^{-1}\|x-z\|} R_j^{\alpha}(x, y)\left(b(x)-b(z)\right) f(y) d \omega(y) \\
&+ \int_{d(x, y)>\varepsilon^{-1}\|x-z\|} \left(R_j^{\alpha}(x, y)-R_j^{\alpha}(z, y)\right) \left( b(z)-b(y)\right) f(y) d \omega(y) \\
& +\int_{d(x, y) \leq \varepsilon^{-1}\|x-z\|} R_j^{\alpha}(x, y)\left(b(x)-b(y)\right) f(y) d \omega(y) \\
& -\int_{d(x, y) \leq \varepsilon^{-1}\|x-z\|} R_j^{\alpha}(z, y) \left(b(z)-b(y)\right) f(y) d \omega(y) \\
=: & \mathcal D_1(f)(x,z)+\mathcal D_2(f)(x,z)+\mathcal D_3(f)(x,z)+\mathcal D_4(f)(x,z).
\end{align*}
Note that for any $z\in B(x,r)$, we have $\|z-x\|<\varepsilon^2$ by the assumed $r<\varepsilon^2$. Since $b \in \Lambda_{d, 0}(\mathbb{R}^N)$, we have $$|b(x)-b(z)|<\|b\|_{\Lambda_{d, 0}(\mathbb{R}^N)}d(x,z)\lesssim \|x-z\|<\varepsilon^2.$$
Now, we begin with estimating $\mathcal D_1(f)(x,z)$. From \eqref{improesti}, we have
\begin{align*}
  | \mathcal D_1(f)(x,z)| \leq & \int_{d(x, y)>\varepsilon^{-1}\|x-z\|}
    \frac{d(x,y)^{\alpha}}{\omega(B(x,d(x,y)))}|b(x)-b(z)| |f(y)| d \omega(y)\\
   \leq & \varepsilon^2 \int_{d(x, y)>\varepsilon^{-1}\|x-z\|}
    \frac{d(x,y)^{\alpha}}{\omega(B(x,d(x,y)))}|f(y)| d \omega(y)\\
   \leq & \varepsilon^2 \int_{\mathbb{R}^{N}}
    \frac{d(x,y)^{\alpha}}{\omega(B(x,d(x,y)))}|f(y)| d \omega(y)\\
   \lesssim &  \varepsilon \|f\|_p^{-\frac{p}{q}+1}
\Big( \sum_{\sigma \in G}Mf(\sigma(x))\Big)^{\frac{p}{q}}.
\end{align*}
Next, for $\mathcal D_2(f)(x,z)$, since $\|x-y\|\geq d(x,y)>\varepsilon^{-1}\|x-z\|$, then we have $\|x-z\|/\|x-y\|<\varepsilon$. By \eqref{smoothcondi2}, we have
\begin{align*}
  | \mathcal D_2(f)(x,z)| \lesssim & \|b\|_{\infty} \int_{d(x, y)>\varepsilon^{-1}\|x-z\|}
   \frac{\|z-x\|}{\|x-y\|} \frac{d(x,y)^{\alpha}}{\omega(B(x,d(x,y)))} |f(y)| d \omega(y)\\
   \leq & \varepsilon  \int_{\mathbb{R}^{N}}
    \frac{d(x,y)^{\alpha}}{\omega(B(x,d(x,y)))}|f(y)| d \omega(y)\\
   \lesssim &  \varepsilon \|f\|_p^{-\frac{p}{q}+1}
\Big( \sum_{\sigma \in G}Mf(\sigma(x))\Big)^{\frac{p}{q}}.
\end{align*}
For $\mathcal D_3(f)(x,z)$, since $d(x, y)\leq \varepsilon^{-1}\|x-z\|<\varepsilon^{-1}\varepsilon^{2}=\varepsilon$, we also have
\begin{align*}
  | \mathcal D_3(f)(x,z)| \lesssim & \|b\|_{\Lambda_{d, 0}(\mathbb{R}^N)}
   \varepsilon^{-1}\|x-z\|
   \int_{d(x, y)\leq \varepsilon^{-1}\|x-z\|}
    \frac{d(x,y)^{\alpha}}{\omega(B(x,d(x,y)))} |f(y)| d \omega(y)\\
   \lesssim & \varepsilon  \int_{\mathbb{R}^{N}}
    \frac{d(x,y)^{\alpha}}{\omega(B(x,d(x,y)))}|f(y)| d \omega(y)\\
   \lesssim &  \varepsilon \|f\|_p^{-\frac{p}{q}+1}
\Big( \sum_{\sigma \in G}Mf(\sigma(x))\Big)^{\frac{p}{q}}.
\end{align*}
For $\mathcal D_4(f)(x,z)$, since $d(x,y)\leq \varepsilon^{-1}\|x-z\|$ and $\varepsilon\in (0,1/4)$, we have $d(z,y)\leq d(z,x)+d(x,y) \leq 5\varepsilon^{-1}\|x-z\|/4$. Then we have
\begin{align*}
  | \mathcal D_4(f)(x,z)| \lesssim & \|b\|_{\Lambda_{d, 0}(\mathbb{R}^N)}
   \varepsilon^{-1}\|x-z\|
   \int_{d(z, y)\leq 5\varepsilon^{-1}\|x-z\|/4}
    \frac{d(z,y)^{\alpha}}{\omega(B(z,d(z,y)))} |f(y)| d \omega(y)\\
   \lesssim & \varepsilon  \int_{\mathbb{R}^{N}}
    \frac{d(z,y)^{\alpha}}{\omega(B(z,d(z,y)))}|f(y)| d \omega(y)\\
   \lesssim & \varepsilon \|f\|_p^{-\frac{p}{q}+1}
(\textstyle\sum_{\sigma \in G}Mf(\sigma(z)))^{\frac{p}{q}}.
\end{align*}

Thus, we can estimate
\begin{align*}
 & \int_{\mathbb{R}^N}  \left|[b,R_j^{\alpha}]f(x)-\left([b,R_j^{\alpha}]f\right)_{B(x, r)}\right|^q d \omega(x)\\
  \lesssim &   \int_{\mathbb{R}^N}\bigg\{\frac{1}{\omega(B(x,r))}\int_{B(x,r)}
  \varepsilon \|f\|_p^{1-\frac{p}{q}}
  \Big[ \Big(\sum_{\sigma \in G}Mf(\sigma(x))\Big)^{\frac{p}{q}}+\Big(\sum_{\sigma \in G}Mf(\sigma(z))\Big)^{\frac{p}{q}} \Big]d\omega(z)\bigg\}^q d \omega(x)\\
  \lesssim &  \|f\|_p^{q-p}\varepsilon^q \int_{\mathbb{R}^N}\bigg\{
 \Big [\Big(\sum_{\sigma \in G}Mf(\sigma(x))\Big)^{\frac{p}{q}}
  +M\Big(\Big(\sum_{\sigma \in G}Mf\Big)^{\frac{p}{q}}\Big)(x)\Big ]\bigg\}^q d \omega(x)\\
  \lesssim &  \|f\|_p^{q-p}\|f\|_p^{p} \varepsilon^{q}
  =   \|f\|_p^{q}
  \varepsilon^{q} \to0,\quad \text{as}\  \    \varepsilon \rightarrow 0,
\end{align*}
as desired.

Thus, the proof of sufficiency is complete.

\smallskip

\subsection{Necessity}
Suppose that $[b, R_j^{\alpha}]$ is compact from $L^p(\mathbb{R}^N, d \omega)$ to $L^q(\mathbb{R}^N, d \omega)$, then $[b, R_j^{\alpha}]$ satisfies condition (a) in Theorem \ref{RieszKolmogorov}. Therefore, by applying Theorem \ref{commutatorbound}, we have $b \in \mathrm{CBMO}_{\text {Dunkl}}(\mathbb{R}^N)$.

Now, we proceed to prove $b \in \mathrm{CVMO}_{\text {Dunkl}}(\mathbb{R}^N)$.
We will use the method of contradiction outlined in \cite{LaceyLi} to achieve this. Let us assume that $b \notin \mathrm{CVMO}_{\text {Dunkl}}(\mathbb{R}^N)$, then we will check that at least one of the three conditions \eqref{eq(1)}--\eqref{eq(3)} in the definition of $\mathrm{CVMO}_{\text {Dunkl}}(\mathbb{R}^N)$ does not hold. Since similar arguments will work for conditions \eqref{eq(1)}--\eqref{eq(3)}, let us suppose that \eqref{eq(1)} does not hold.

 Suppose that there exists some $\delta_0>0$ and a sequence of balls $\{B_i\}_{i \in I}$ where $B_i:=B(x_0^i,r_i)$ and $0\in B_i$ for each $i$. We also assume that $r_i$ satisfies $r_i \rightarrow 0$ as $i \rightarrow \infty$ and
$$
\frac{1}{\omega(B_i)} \int_{B_i}|b(x)-b_{B_i}| d \omega(x) > \delta_0.
$$

We can choose $\{B_l\}_{l \in I} \subseteq \{B_i\}_{i \in I}$ with
\begin{equation}\label{radius decay}
10 r_{l+1} \leq r_{l}.
\end{equation}

Note that
for each $B_i=B(x_0^i, r_i)$, we choose $\tilde{B}_i=B(y_0^i, r_i)$ such that $\|y_0^i-x_0^i\|=5r$, and for any $(x,y )\in  B_i\times\tilde{B}_i$, we have $y_j-x_j \geq r$ and $\|x-y\| \approx r$. For $\tilde{B}_i$, we can define a median value of $b$ on the such a ball $\tilde{B}_i$, denoted by $m_b(\tilde{B}_i)$. Then we have two sets below
$$
F_{i, 1} \subseteq \left\{y \in \tilde{B}_i: b(y) \leq m_b (\tilde{B}_i ) \right\}, \quad F_{i, 2} \subseteq \left\{y \in \tilde{B}_i: b(y) \geq m_b (\tilde{B}_i ) \right\},
$$
which have a measure at least $ \omega(\tilde{B}_i)/2$.

Similar to the argument in Section \ref{Seclowerbound}, we also define the sets
$$
E_{i, 1} \subseteq \left\{x \in B_i: b(x) \geq m_b (\tilde{B}_i ) \right\}, \quad
E_{i, 2} \subseteq \left\{x \in B_i: b(x)  <   m_b (\tilde{B}_i ) \right\}.
$$
Then, $B_i=E_{i, 1} \bigcup E_{i, 2}$ and $E_{i, 1} \bigcap E_{i, 2}=\emptyset$. For $(x, y) \in \left(E_{i, 1} \times F_{i, 1} \right)\bigcup \left(E_{i, 2} \times F_{i, 2}\right)$, we also have that $b(x)-b(y)$ does not change sign and
$$
|b(x)-b(y)|\geq \left|b(x)-m_b (\tilde{B}_i )\right |.
$$
Define the following sets
$$
\tilde{F}_{i, 1}:=F_{i, 1} \backslash \bigcup_{l=i+1}^{\infty} \tilde{B}_l \quad  \text{and}\quad  \tilde{F}_{i, 2}:=F_{i, 2} \backslash \bigcup_{l=i+1}^{\infty} \tilde{B}_l \quad \text{for}\   \   i=1,2, \ldots .
$$
Then, it follows from \eqref{radius decay} that
\begin{align*}
\omega(\tilde{B}_i ) \geq
\omega(\tilde{F}_{i, 1})  \geq
\frac{1}{6} \omega(\tilde{B}_i ) \quad
 \text{ and }\quad
\omega(\tilde{B}_i ) \geq
\omega(\tilde{F}_{i, 2})  \geq \frac{1}{6} \omega(\tilde{B}_i ).
\end{align*}
Now, we have
\begin{align*}
\delta_0 < & \frac{1}{\omega(B_i)} \int_{B_i} \left|b(x)-b_{B_i}\right| d \omega(x) \leq \frac{2}{\omega(B_i )} \int_{B_i}\left  |b(x)-m_b (\tilde{B}_i ) \right | d \omega(x) \\
= & \frac{2}{\omega(B_i)} \int_{E_{i, 1}} \left|b(x)-m_b(\tilde{B}_i)\right| d \omega(x)
 +\frac{2}{\omega(B_i )} \int_{E_{i, 2}}  \left|b(x)-m_b (\tilde{B}_i)\right| d \omega(x)
\end{align*}
Then we can deduce that at least one of the following inequalities holds:
$$
\frac{2}{\omega (B_i )} \int_{E_{i, 1}} \left|b(x)-m_b (\tilde{B}_i ) \right| d \omega(x) \geq \frac{\delta_0}{2} , \quad \frac{2}{\omega (B_i )} \int_{E_{i, 2}} \left|b(x)-m_b (\tilde{B}_i ) \right| d \omega(x) \geq \frac{\delta_0}{2}.
$$
Without loss of generality, we may assume that the first one holds. Then we have that
\begin{align*}
\frac{\delta_0}{4} & \leq \frac{1}{\omega (B_i )} \int_{E_{i, 1}} \left|b(x)-m_b (\tilde{B}_i ) \right| d \omega(x) \\
& \lesssim \frac{1}{\omega(B_i )} \frac{1}{\omega (B_i )}
\int_{\tilde{F}_{i, 1}}\int_{E_{i, 1}} \left|b(x)-m_b (\tilde{B}_i ) \right| d \omega(x)d\omega(y) \\
& \lesssim \frac{r_i^{-\alpha}}{\omega (B_i )} \int_{E_{i, 1}}
\int_{\mathbb{R}^N} \frac{r_i^{\alpha}}{\omega (B(x,r_i) )} |b(x)-b(y)|  \mathsf{1}_{\tilde{F}_{i, 1}}(y)  d \omega(y) d \omega(x)\\
& \lesssim \frac{r_i^{-\alpha}}{\omega (B_i )^{1/p'}} \int_{E_{i, 1}}
\bigg| [b, R_j^{\alpha}]\,
\bigg (\frac{\mathsf{1}_{\tilde{F}_{i, 1}}}{\omega (B_i )^{1/p}}\bigg )\,(x) \bigg|d \omega(x)\\
& \lesssim \frac{r_i^{-\alpha}}{\omega (B_i )^{1/p'}} \omega(E_{i, 1})^{1/q'}
\bigg \|[b, R_j^{\alpha}]\,
 \frac{\mathsf{1}_{\tilde{F}_{i, 1}}}{\omega (B_i )^{1/p}} \bigg \|_{q}\\
& \lesssim
\left\|[b, R_j^{\alpha}]f_i\right\|_{q},
\end{align*}
where $f_i(x) := \mathsf{1}_{\tilde{F}_{i, 1}}(x)\omega (B_i )^{-1/p}$. Note that $f_i$ has disjoint support for different $i$ and $\|f_i\|_p\approx 1$.

Let us consider $\psi$ in the closure of $\left\{[b, R_j^{\alpha}]f_i\right\}_i$, then we have $\|\psi\|_{q} \gtrsim 1$. Now choose a subsequence $\{f_{i_k}\}_i$ such that
\begin{align}\label{psi}
\left\|\psi-[b, R_j^{\alpha}]f_{i_k}\right\|_{q} \leq 2^{-k}.
\end{align}
To complete the proof, we choose a non-negative numerical sequence $\{c_k\}_{k=1}^{\infty}$ such that
$$
\left\{\begin{array}{l}
c_1=1 ; \\
c_k=\mathfrak{n} 2^{-\mathfrak{n}}, \quad 2^{\mathfrak{n}} \leq k \leq 2^{\mathfrak{n}+1}-1
\end{array}\right.
$$
for $\mathfrak{n}=1,2, \ldots$.

By the calculations in \cite{HLLW},
we know that $\displaystyle\|\{c_k\}\|_{l^q},\|\{c_k\}\|_{l^{q^{\prime}}}<\infty$ for $q>1$ but $\displaystyle\|\{c_k\}\|_{l^1}=\infty$.
Moreover, $\|\phi\|_{p} <\infty$ with
$\phi=\sum_{k=1}^{\infty} c_k f_{i_k}$.

For any $\phi \in L^p (\mathbb{R}^N, d \omega )$,
by H\"{o}lder's inequality and \eqref{psi},
we have
\begin{align*}
 &\Big\|\sum_{k\geq 1} c_k \psi- [b, R_j^{\alpha} ] \phi \Big\|_{q}  \\
 \leq & \Big\|\sum_{k\geq 1} c_k \left (\psi- [b, R_j^{\alpha} ] f_{i_k}  \right ) \Big\|_{q} \\
\leq &\|c_k \|_{l^{q^{\prime}}} \Big[\sum_{k\geq 1}  \left \|\psi- [b, R_j^{\alpha} ] f_{i_k}  \right \|_{q}^q \Big]^{1/q} \\
\lesssim & 1 .
\end{align*}

Hence we conclude that $\sum_k c_k \psi \in L^q(\mathbb{R}^N, d \omega)$, but $\sum_k c_k \psi$ is infinite on set of positive measure. This leads to a contradiction. Thus, we complete our proof.

\section*{Acknowledgements}
Yanping Chen was supported by the National Natural Science Foundation of China (Grant numbers [12371092], [12326366] and [12326371]). Liangchuan Wu is supported by NNSF of China (Grant number 12201002), ARC DP220100285, Anhui NSF of China (Grant number 2208085QA03) and Excellent University Research and Innovation Team in Anhui Province (Grant number 2024AH010002).

\end{document}